\documentclass[english]{amsart}

\usepackage[all,cmtip]{xy}

\usepackage{amsmath,amssymb,amscd,amsfonts}

\usepackage{babel}
\usepackage{amstext}
\usepackage{amsmath}
\usepackage{amsfonts}
\usepackage{latexsym}
\usepackage{ifthen}

\usepackage[all,cmtip]{xy}
\xyoption{all}
\newcommand{\cI}{{\mathcal I}}

\newcommand{\codim}{{\rm codim}}
\newcommand{\Pic}{{\rm Pic}}

\newcommand\cE{{\mathcal E}}

\newcommand\cF{{\mathcal F}}

\newcommand\bZ{{\mathbb Z}}
\newcommand\bC{{\mathbb C}}
\newcommand\bQ{{\mathbb Q}}

\newcommand\bN{{\mathbb N}}

\newcommand\cO{{\mathcal O}}

\def\Re{\mathop{\rm Re}\nolimits}

\def\rank{\mathop{\rm rank}\nolimits}

\def\Alb{\mathop{\rm Alb}\nolimits}

\def\Pic{\mathop{\rm Pic}\nolimits}

\def\tor{\mathop{\rm tor}\nolimits}
\def\Div{\mathop{\rm div}\nolimits}
\def\Cod{\mathop{\rm codim}\nolimits}
\def\pr{\mathop{\rm pr}\nolimits}

\def\Supp{\mathop{\rm Supp}\nolimits}
\def\dbar{\overline\partial}
\def\ddbar{\partial\overline\partial}
\def\cO{{\mathcal O}}

\def\cE{{\mathcal E}}

\def\cF{{\mathcal F}}

\let\ol=\overline
\let\ep=\varepsilon
\let\wt=\widetilde
\let\wh=\widehat
\def\bQ{{\mathbb Q}}
\def\bC{{\mathbb C}}

\def\bZ{{\mathbb Z}}
\def\bN{{\mathbb N}}

\newcounter{lemma}
\renewcommand{\thelemma}{\strut\kern-3pt\arabic{section}.\arabic{lemma}}
\newtheorem{lemma1}[lemma]{\setcounter{equation}{0}}
\let\saveref=\ref
\def\ref#1{\strut\kern3pt{\saveref{#1}}}
\def\eqref#1{({\saveref{#1}})}

\newenvironment{lemma}{\begin{lemma1}{\bf Lemma.}}{\end{lemma1}}

\newenvironment{theorem}{\begin{lemma1}{\bf Theorem.}}{\end{lemma1}}
\newenvironment{proposition}{\begin{lemma1}{\bf Proposition.}}{\end{lemma1}}
\newenvironment{corollary}{\begin{lemma1}{\bf Corollary.}}{\end{lemma1}}
\newenvironment{remark}{\begin{lemma1}{\bf Remark.}\rm}{\end{lemma1}}

\newenvironment{claim}{\begin{lemma1}{\bf Claim.}}{\end{lemma1}}

\begin{document}
\title[]{Kodaira dimension of
algebraic fiber \\ spaces over abelian varieties}

\author{Junyan Cao and Mihai P\u aun}
\address{Junyan CAO, Universit\'e Paris 6 \\
Institut de Math\'{e}matiques de Jussieu\\
4, Place Jussieu, Paris 75252, France }
\email{junyan.cao@imj-prg.fr}

\address{Mihai P\u{a}un, Korea Institute for Advanced Study\\
85 Hoegiro, Dongdaemun-gu\\
Seoul 130-722, South Korea,}
\email{paun@kias.re.kr}
\date{\today}
\thanks{}

\begin{abstract} 
In this article we provide a proof of the Iitaka $C_{nm}$ conjecture for 
algebraic fiber spaces over tori. 
\end{abstract}
\maketitle

\section{Introduction}

\noindent Let $p: X\to Y$ be an algebraic fiber 
space, i.e. $X$ and $Y$ are non-singular projective manifolds and
$p$ is surjective with connected fibers. An important problem in birational geometry is 
the \emph{Iitaka conjecture}, stating that 
\begin{equation}\label{eq40}
\kappa(X)\geq \kappa(Y)+\kappa(X/Y)
\end{equation}
where $\kappa(X)$ is the Kodaira dimension of $X$, and $\kappa(X/Y)$ is the Kodaira dimension of 
a general fiber of $p$. 
\smallskip

In this article our goal is to show that \emph{the log-version of the inequality {\rm \eqref{eq40}} holds true, provided that
the base $Y$ is an abelian variety}; this generalizes the result obtained by Y.~Kawamata in 1982, cf.
\cite{Kawa82}. More precisely, our main theorem states as follows.

\begin{theorem}\label{jcmp}
Let $p: X\rightarrow A$ be an algebraic fiber space, where $A$ is an Abelian variety.
Let $\Delta$ be an effective $\bQ$-divisor such that the pair $(X, \Delta)$ is klt, and let $F$ be a generic fiber of $p$.
Then
\begin{equation}\label{neweqfin}
\kappa (K_X +\Delta) \geq \kappa (K_F +\Delta_F), 
\end{equation}
where $\Delta_F = \Delta|_F$.
\end{theorem}

We will give next a few hints about the 
proof, so as to situate our work in     
the impressive body of papers dedicated to this problem, cf. \cite{Caucher_1}, \cite{CauChen}, \cite{CH02},
\cite{CH09}, \cite{CH11}, \cite{Fujino}, \cite{Hacon}, \cite{Kawa82b}, \cite{konnar}, \cite{kon_pet}, 
\cite{Lai}, \cite{PS}, \cite{Tsu10}, \cite{Vie81}, \cite{Vie95} among many others.
The key ingredient of their proof is the positivity of direct 
image sheaves $\displaystyle p_\star(mK_{X/Y}+m \Delta)$ (notable exceptions to this statement are the works of 
Hacon-Chen cf. \cite{CH11} and Birkar-Chen \cite{CauChen}, respectively), where $m\in \bN$ such that $m\Delta$ is a line bundle.
Up to a certain point, the arguments presented here follow this \emph{main stream}.
\smallskip

\noindent Roughly speaking, we only have to deal with the following two extreme cases: either
the determinant of the direct image $\displaystyle p_\star(mK_{X/Y}+m\Delta)$ is big, or it is topologically trivial.
This is obtained as consequence of Theorem \ref{torus} in section 3.

In the first case we show that the relative canonical bundle $K_{X/Y}+\Delta$ is greater than 
the $p$-inverse image of an  ample divisor, up to a divisor whose image has codimension at 
least two in $Y$. This will allow us to 
extend pluricanonical sections from the fibers of $p$, and therefore proves Theorem \ref{jcmp}
in the case under discussion. The main result involved in this part of our proof is Theorem~\ref{positivelemma}, which is nothing but
a generalization of E. Viehweg's \emph{weak semistability} results, cf. \cite{Tsu10} \cite{Vie95}.

If the determinant of the direct image is flat, then we consider the direct image 
sheaf
$$p_\star(mK_{X/Y}+ m\Delta),$$
where $m$ is as above.
General results show the existence of a subset $\Sigma_0\subset Y$ 
whose codimension is at least two, such that the restriction 
$$\cE:= \displaystyle p_\star(mK_{X/Y}+ m\Delta)|_{Y\setminus \Sigma_0}$$
corresponds to a vector bundle. Now, if the vector bundle $\cE$ is e.g. trivial, then 
it is a simple matter to extend the pluricanonical sections  
$\displaystyle mK_{X/Y}+ m\Delta|_{X_y}$ defined on a general enough fiber $X_y$.
Of course, a priori we cannot expect this to happen; nevertheless,
we 
show that there exists 
an open set $\Sigma\subset Y$ whose codimension is at least two, such that 
$$\cE|_{Y\setminus \Sigma}$$ is a 
Hermitian flat vector bundle when endowed with a canonical $L^2$ metric (the so-called ``Narasim\-han-Simha").
This relies on the metric 
properties of the sheaves  $\displaystyle p_\star(mK_{X/Y}+m\Delta)$ (for which we refer to \cite{PT}), combined with a result of H.~Raufi concerning the existence of the curvature current corresponding to  
a singular Hermitian metric (see the results in \cite{Raufi}). 

We remark that the classification of flat vector bundles on elliptic curves
used in \cite{Kawa82} is replaced in this paper by the following considerations. 
The bundle $\cE$ considered above 
gives a unitary representation $\rho$ of rank $r$ of $\displaystyle \pi_1(A\setminus \Sigma)$. 
As we have already mentioned, the analytic set $\Sigma$ has codimension at least 2, so the 
fundamental group of its complement in $A$ equals $\pi_1(A)$, which is a free abelian group. 
If the image of $\rho$ is finite, then an appropriate power of any 
section of the restriction $\displaystyle (mK_{X/Y}+ m\Delta)|_{X_y}$
extends to $X\setminus p^{-1}(\Sigma)$. Moreover, the resulting section has finite $L^2$ norm, so it extends across
$p^{-1}(\Sigma)$. 
Next, a unitary representation of any free abelian group splits as a direct sum of representations of rank one. 
If the image of the representation $\rho$ is infinite, then one of the factors of the splitting will have the same property. In this case we conclude by 
using another crucial result
due Campana-Peternell cf. \cite[Thm 3.1]{CP11}, as well as the generalization in \cite{CKP12}. 

\medskip
 
\noindent Our article is organized as follows. In section two we will collect
some basic facts concerning the construction of metrics on relative pluricanonical bundles and 
their direct images, as well as a few results concerning singular Hermitian metrics on vector bundles. The main result we establish in section three is Theorem~\ref{positivelemma}; the 
techniques needed to prove it refine the arguments of E. Viehweg and H. Tsuji, among many others. 
The proof of the inequality \eqref{neweqfin} is completed in section four.  
As a complement to the techniques and results we obtain in this article, we establish in section five 
a version of \eqref{neweqfin} for an arbitrary algebraic fiber space $p:X\to Y$ for which the line bundle
$\det\displaystyle p_\star(mK_{X/Y}+m\Delta)$
is topologically trivial.

\medskip
\noindent {\bf Acknowledgements.} We owe a debt of gratitude to Bo Berndtsson, 
S\'ebastien Boucksom, Fr\'{e}d\'{e}ric Campana, Philippe Eyssidieux, Christopher Hacon, Andreas H\"oring, Zhi Jiang, Yujiro Kawamata,
Mihnea Popa, Hossein Raufi, Christian Schnell and Shigeharu Takayama for sharing generously with us their results and intuitions on the topics analyzed here. It is our pleasure 
to acknowledge the partial support we have benefited from the ANR project ``MACK" 
during the preparation of the present article. Last but not least, we would like to thank the anonymous referee
for constructive criticism and excellent suggestions who helped us to improve substantially the quality of this work.

\section{Positivity of direct images: a few techniques and results}

\medskip

\noindent The positivity results for direct images of twisted pluricanonical relative 
bundles are part of the main tools in our proof. 
 In this section we recall the construction of the Bergman metric, and some of its properties;
 we will equally collect a few results taken from \cite{BP1}, \cite{BP10}, \cite{PT}. Even if our ``language" is mostly analytic, a large part of the results here have counterparts/versions in algebraic geometry, cf. \cite{PS} and the references therein. 
 \smallskip  
 \subsection{The relative Bergman metric}
 
 \noindent Let $X$ and $Y$ be two projective manifolds, which are assumed to be non-singular.
 Let $p: X\to Y$ be a surjective map, and let $(L, h_L)\to X$ be a line bundle endowed with a Hermitian 
 metric $h_L$. We make the convention that unless explicitly mentioned otherwise, the metric in this article are allowed to be singular. As part of the set-up, we assume that we have
 \begin{equation}\label{equa10}
 \Theta_{h_L}(L)\geq 0
 \end{equation}
 in the sense of currents on $X$. By definition, this means the following: let $\Omega\subset X$ be 
 any trivialization subset for $L$, such that the restriction $h_L|_{\Omega}$ corresponds to the 
 metric $\vert\cdot \vert^2e^{-\varphi_L}$. Then \eqref{equa10} requires that $\varphi_L$ is psh, so that we have 
 $\sqrt{-1}\ddbar \varphi_L\geq 0$ in the sense of currents. 
 
 \noindent 
In this context we recall  the construction of the Bergman 
metric $\displaystyle e^{-\varphi_{X/Y}}$ on the bundle $K_{X/Y}+ L$; we refer to \cite{BP1} for further details. 

\noindent Let $Y_0$ be a Zariski open subset of $Y$ such that $p$ is smooth over $Y_0$, and for every $y\in Y_0$, the fiber $X_y$ satisfies 
$h^0 (X_y, K_{X/Y}\otimes L\otimes \mathcal{I} (h_L |_{X_y})) = \rank p_\star (K_{X/Y}\otimes L\otimes \mathcal{I} (h_L))$.
Let $X^0$ be the $p$-inverse image of $Y_0$ and let $x_0\in X^0$ be an arbitrary point; let $z^1,\dots, z^{n+m}$ be local coordinates centered at $x_0$, 
and let $t^1,\dots , t^m$ be a coordinate
centered at $y_0:= p(x_0)$. We can assume that $z^{n+j} =p^\star (t_j)$ for every $j$.
We consider as well a trivialization of $L$ near $x_0$. 
With this choice of local coordinates, we have a local trivialization of the tangent bundles of $X$ and $Y$ respectively, and hence of the (twisted) relative canonical bundle. 

\noindent The local weight of the metric $\displaystyle e^{-\varphi_{X/Y}}$ with respect to this is given by the equality

\begin{equation}\label{relative}
e^{\varphi_{X/Y}(x_0)}= \sup_{\Vert u\Vert_{y_0}\leq 1} |F_u (x_0)|^2
\end{equation}
where the notations are as follows: $u$ is a section of $\displaystyle K_{X_{y_0}}+ L|_{X_{y_0}}$, 
and $F_u$ is the coefficient of $dz^1\wedge \dots \wedge dz^{n+m}$
in the local expression of $u\wedge p^\star dt$. 
The norm which appears in the definition \eqref{relative} is obtained by the fiber integral
\begin{equation}\label{equa121}
\Vert u\Vert_{y_0} ^2:= \int_{X_{y_0}} |u|^2e^{-\varphi_L}.
\end{equation}
\medskip

\noindent An equivalent way of defining \eqref{relative} is via an orthonormal basis, say $u_1,
\dots , u_k$ of sections of $\displaystyle K_{X_{y_0}}+ L|_{X_{y_0}}$. Then we see that
\begin{equation}\label{on}
e^{\varphi_{X/Y}(x_0)}= \sum_{j=1}^N |F_j(x_0)|^2
\end{equation}
where $F_j$ are the functions corresponding to $u_j$. 

\noindent The Bergman metric $\displaystyle h_{X/Y}= e^{-\varphi_{X/Y}}$ can also be introduced in an
 intrinsic manner as follows. Let $\xi$ be a vector in the fiber over $x_0$ of the line bundle 
 $\displaystyle -(K_{X/Y}+ L)_{x_0}$.
 The we have
 \begin{equation}\label{equa11}
 \vert \xi\vert^2= \sup_{\Vert u\Vert_{y_0}\leq 1} |\langle \xi, u_{x_0}\rangle |^2.
 \end{equation}
This defines a metric on the dual bundle, which induces $h_{X/Y}$ on $K_{X/Y}+ L$.
\medskip

\noindent As we see from \eqref{on}, the restriction of the metric 
$e^{\varphi_{X/Y}}$ to the fiber $X_{y_0}$ coincides with the metric induced by any \emph{orthonormal basis}
of the space of holomorphic sections of $\displaystyle K_{X_{y_0}}+ L|_{X_{y_0}}$. Hence the variation from one fiber to another is in general a $\mathcal C^\infty$ operation, since the said orthonormalization process is involved. Thus it is a remarkable fact that this metric has positive curvature in the sense of currents on $X$.

\begin{theorem} {\rm (\cite[Thm 0.1]{BP1})}\label{rel1} The curvature of the metric $h_{X/Y}$ on the twisted relative canonical bundle $\displaystyle K_{X/Y}+ L|_{X^0}$ is positive in the sense of currents. Moreover, the local weights $\varphi_{X/Y}$ are uniformly bounded from above on $X^0$, so they admit a unique extension as psh functions.
\end{theorem}

\begin{remark}\label{extensionremark}
The fact that the uniform boundness of $\varphi_{X/Y}$ on $X_0$ 
implies that it admits a unique extension is a standard result 
in pluripotential theory (cf.\ \cite{klimek}, page 52 and the references therein, as well as
\cite{Dem}, page 43-44) 
which we briefly recall now.
Let $X$ be a complex manifold (not necessarily compact) and let $Z$
be a complex subvariety in $X$. Let $\varphi$ be a psh function defined 
on $X\setminus Z$. The following assertions hold true. 
\smallskip

\noindent{\rm (i):} If $\codim_X (Z) \geq 2$, then $\varphi$ admits a unique extension as 
a psh function on $X$. 
\smallskip

\noindent{\rm (ii):} If $\codim_X (Z) \geq 1$ and $\varphi$ is uniformly bounded from above on $X\setminus Z$, then $\varphi$ admits a unique extension as 
a psh function on $X$.

\noindent We will use these two properties frequently in the article. 
\end{remark}

The definition \eqref{relative}, although not intrinsically formulated, is explicit 
enough so as to imply the following statement. Let $p:X\to Y$ be a dominant map, 
such that $X$ is K\"ahler; we denote by 
$\Delta$ the analytic set corresponding to the critical values of $p$. We assume that $\Delta$ is a snc divisor of $Y$, and we also assume that the $p$-inverse image of $\Delta$ equals
\begin{equation}\label{equa221}
\sum_{i\in I} e_i W_i
\end{equation}
where $e_i$ are positive integers, and $W_i$ are reduced hypersurfaces of $X$. 
\medskip

\noindent The next statement can be seen as a metric version of the corresponding results due to 
Y. Kawamata in \cite{Kawa98} and F. Campana in \cite{Ca04}, respectively.
\begin{theorem}\label{nonreduce}
Let $\Theta_{X/Y}$ be the curvature current corresponding to the Bergman metric {\rm \eqref{relative}}. Then we have 
\begin{equation}\label{equa222}
\Theta_{X/Y}\geq [\Sigma_p]:= \sum_{i\in I_h}(e_i-1)[W_i]
\end{equation}
in the sense of currents on $X$ where $I_h$ is the set of indexes $i\in I$ such that $p(W_i)$ is a divisor of $Y$. In particular,   
the current $\Theta_{X/Y}$ is singular along the multiple fibers of the map $p$.
\end{theorem}

\begin{proof}
Let $x_0\in W_1$ be a non-singular point of one of the sets appearing in \eqref{equa221}. 

We consider a coordinate set $\Omega$ containing the point $x_0$, and we fix the coordinates 
$\displaystyle (z_1,\dots , z_{n+m})$ on $\Omega$, such that $W_1\cap \Omega= (z_{n+ 1}= 0)$. The local structure of the map $p$ is as follows
\begin{equation}\label{eq1}
\big(z_1,\dots, z_{n+m})\to (z_{n+1}^{e_1}, z_{n+2},\dots, z_{n+ m}\big),
\end{equation}
so that we assume implicitly that $p(W_1)$ is given locally by $t_1=0$. 

The intersection $p^{-1}(t)\cap \Omega$ of the fibers of $p$ with the coordinate set $\Omega$ 
can be identified with an open set in $\bC^n$.  This allows us to bound the absolute value of the quantity which computes the Bergman metric locally at $x_0$, as we see next.

Let $t\in Y\setminus (t_1= 0)$ be a point near $p(x_0)$, and let $u$ be a section of the $\displaystyle K_{X_{t}}+ L|_{X_{t}}$ as in \eqref{equa121}. If $\|u\|_{X_{t}} ^2=1$, then
by the construction of $F_u$, we have
$$\int_{X_{t} \cap \Omega} \frac{|F_u|^2}{\left| z_{n+1}\right|^{2e_1 -2}}d\lambda 
 \leq \|u\|_{X_{t}} ^2 =1,$$
 where $d\lambda$ is the Lebesgue measure with respect to $z_1,
 \dots, z_n$. 
Combining this with \eqref{relative}, we have thus
\begin{equation}\label{eq2002}
\varphi_{X/Y}(z)\leq (e_1 -1)\log\vert z_{n+ 1}\vert^2+ O(1),
\end{equation}
where the quantity $O(1)$ in \eqref{eq2002} is uniform with respect to $z\in \Omega\setminus W_1$.
This shows that the Lelong number of $\varphi_{X/Y}$ at any generic point of $W_1$ is greater than $e_1 -1$.
As a consequence, we have
$$\Theta_{X/Y} \geq  (e_1 -1)[W_1],$$
and the proof is finished.
\end{proof}
\medskip


\bigskip

\noindent The construction of the metric $h_{X/Y}$ has a perfect pluricanonical analogue, as we recall next.
Let $u$ be a section of the bundle $\displaystyle mK_{X_y}+ L$, where $m\geq 1$ is a positive integer. 
Then we define 
\begin{equation}\label{eq41}
\Vert u\Vert^{\frac{2}{m}}_y:= \int_{X_y}\vert u\vert^{\frac{2}{m}}e^{-\frac{1}{m}\varphi_L},
\end{equation}
and the definition \eqref{equa11} generalizes immediately, as follows. Let $\xi$ be a vector in the fiber 
over $x$ of the dual bundle 
 $\displaystyle -(mK_{X/Y}+ L)_{x}$.
Then we have
\begin{equation}\label{eq44}
 \vert \xi\vert^2= \sup_{\Vert u\Vert_y\leq 1} |\langle \xi, u_{x}\rangle |^2.
\end{equation}
We denote the 
resulting metric by $h^{(m)}_{X/Y}$.
\medskip

\noindent We recall next the analogue of Theorem \ref{rel1}, as follows.

\begin{theorem} {\rm (\cite[Thm 0.1]{BP10})}\label{rel2} The curvature of the metric $\displaystyle h_{X/Y}^{(m)}$ 
on the twisted relative pluricanonical bundle $\displaystyle mK_{X/Y}+ L|_{X^0}$ is positive in the sense of currents. 
Moreover, the local weights $\varphi_{X/Y}$ are uniformly bounded from above on $X^0$, so they admit a unique extension as psh functions.
\end{theorem}

\begin{remark}\label{msing} {\rm
If the map $p$ verifies the hypothesis of Theorem \ref{nonreduce}, then we infer that 
\begin{equation}
\Theta_{h_{X/Y}^{(m)}}\big(mK_{X/Y}+ L\big)\geq m[\Sigma_p].
\end{equation}
The proof is the same as in Theorem \ref{nonreduce}: if the local structure of the map $p$ is as in \eqref{eq1}, then the 
$L^{2/m}$ normalization bound for the sections involved in the computation of the metric $h_{X/Y}^{(m)}$
imply that the local pointwise norm of these sections is bounded. The weights of the metric 
 $\displaystyle h_{X/Y}^{(m)}$ are given by the wedge product with $dt^{\otimes m}$, so the conclusion follows.
}
\end{remark}

\bigskip





\subsection{Singular metrics on vector bundles and direct image sheaves} 
We recall first the definition of singular Hermitian metrics on vector bundles cf. \cite{BP1}, \cite{Raufi} and \cite{Pau16}.
Let $E\rightarrow X$ be a holomorphic vector bundle of rank $r$ on a complex manifold $X$. We denote by 
$$H_r := \{A =(a_{i, \overline{j}})\}$$
the set of $r\times r$, semi-positive definite Hermitian matrices. Let $\overline{H}_r$ be the space of semi-positive, possibly unbounded Hermitian
forms on $\mathbb{C}^r$. A singular Hermitian metric $h$ on $E$ is given locally by a measurable map with values in $\overline{H}_r$
such that 
$$0 < \det h < +\infty $$
almost everywhere. In the definition above, a matrix valued function $h= (h_{i, \overline{j}})$ is measurable provided that all entries 
$h_{i, \overline{j}}$ are measurable.

This notion is somehow too general; in particular, it is impossible to define a
curvature current corresponding to it, as soon as the rank of $E$ is at least two, see \cite{Raufi}
for a clear example illustrating this. Nevertheless, as observed in \cite{BP1}, one can still formulate the notion of 
negativity/positivity in the sense of Griffiths: the bundle $(E, h_E)$ is \emph{negatively curved} if 
\begin{equation}\label{eq30}
x\to \log\vert u\vert_{h_E, x}^2
\end{equation}
is psh, for any choice of a holomorphic local section $u$ of $E$. The bundle $(E, h_E)$ is 
\emph{positively curved} 
if $(E^\star, h_E^\star)$ is negative.
\smallskip

\noindent It is important to notice the following consequence of the Griffiths negativity  
a\-ssumption for a singular Hermitian vector bundle $(E, h_E)$. Let $\xi$ be a local holomorphic section 
of $E$
defined on a coordinate 
open set $U\subset X$. Since the function
$$\vert \xi\vert^2= \sum_{\alpha, \beta} \xi_\alpha\ol{\xi_\beta} h_{\alpha\ol{\beta}}$$
is psh, in particular it follows that it is unambiguously defined at each point of $U$. We infer that the same is true for the coefficients 
$\big(h_{\alpha\ol\beta}\big)$. Moreover, the function $\vert \xi\vert^2$ is bounded from above on any 
relatively compact $U^\prime\Subset U$ so it follows that we have
\begin{equation}\label{eq42}
\sup_{z\in U^\prime}|h_{\alpha\ol\beta}(z)\vert\leq C.
\end{equation}

\medskip

\noindent The following result is a particular case of \cite{Raufi}; it gives a sufficient criteria 
in order to define 
the notion of \emph{curvature current} associated to $(E, h_E)$ which fits perfectly 
to what we will need later on in the paper. 

\begin{theorem} {\rm (\cite[Thm 1.6]{Raufi})}\label{singsing} 
Let $(E, h_E)$ be a positively curved singular Hermitian vector bundle of rank $r$. We assume that the induced metric 
$\det h_E$ on the determinant $\Lambda^r E$ of $E$ is non-singular. Then the coefficients of the Chern connection 
form $\theta_E:= h_E^{-1}\partial h_E$ belong to $L^2_{\rm loc}$. As a consequence, the curvature
current $\displaystyle \Theta_{h_E}(E)$ is well defined and it is moreover positive in the sense of Griffiths. 
Moreover, it can be written locally as $\dbar \theta_E$.  
\end{theorem}
\medskip

\noindent We provide here a few explanations about the statement \ref{singsing}. The fact that 
$\displaystyle \Theta_{h_E}(E)$ is well defined as matrix-valued $(1,1)$-current means that locally on 
some coordinate set $U$ centered at some point
$x\in X$ we have
\begin{equation}\label{singcurv}
\Theta_{h_E}(E)|_U= \sum_{j, k, \alpha,\beta}\mu_{j\ol k\alpha\ol \beta}dz^j\wedge dz^{\ol k}e_\alpha\otimes e_\beta^\star
\end{equation}
where $\displaystyle \mu_{j\ol k\alpha\ol \beta}$ are \emph{measures} on $U$ (rather than smooth functions as in the 
classical case), $(e_\alpha)_{\alpha=1,\dots , r}$ is a local holomorphic frame of $E$ and $(z^i)_{i=1,\dots ,n}$ are local coordinates. 
The positivity in the sense of Griffiths we are referring to in Theorem \ref{singsing}
means that for any local holomorphic vector field $\displaystyle \sum v^j\frac{\partial}{\partial z^j}$ and 
for any local holomorphic section $\displaystyle \sum \xi^\alpha e_\alpha$, the measure
\begin{equation}\label{measure}
\sum \mu_{j\ol k\alpha\ol \beta}v^j\ol{v^k}\xi^\alpha\ol {\xi^\beta}
\end{equation}
is (real and) positive on $U$. The positivity of the measure \eqref{measure} is obtained by H. Raufi in
\cite{Raufi} by using an approximation procedure: he shows that under the hypothesis of Theorem \ref{singsing},
locally near each point of $X$
there exists a sequence of non-singular metrics $h_{E, k}$ such that $(E, h_{E, k})$ is Griffiths-positively curved 
(in the usual sense), and that the corresponding curvature form is converging to $\dbar \theta_E$.

\medskip

\noindent We will apply next this result in the context of direct images of twisted (pluri)ca\-nonical bundles.
The set-up is the same as in the previous subsection; let $Y_1\subset Y$ be the intersection of the 
set of regular values of $p$ with the maximal subset of $Y$ on which the direct image 
sheaf $p_{\star}(K_{X/Y}+L)$ is locally free.

\noindent The fiberwise canonical $L^{2}$-metrics 
\begin{equation}\label{equa12}
	g_{1,y}(u,u):=\int_{X_{y}} |u|^2 e^{-\varphi_L} \le +\infty
\end{equation}
for $u \in H^{0}(X_{y},K_{X_{y}}+L_{y})$ induces a singular Hermitian metric 
$g_{X/Y}$ on the bundle $\displaystyle p_{\star}(K_{X/Y}+L)|_{Y_1}$ whose curvature is positive (cf. \cite{Vie95}). 
The following result gives an important precision concerning this framework.

\begin{theorem}{\rm(\cite[Thm 3.3.5]{PT})}\label{ext} 
We suppose that the natural inclusion 
\begin{equation}\label{equa13}
p_{\star}\big((K_{X/Y}+L)\otimes {\mathcal I}(h_L)\big) \subset p_{\star}(K_{X/Y}+L)
\end{equation} is generically isomorphic.
Then the canonical $L^{2}$-metric $g_{X/Y}$ on the direct image 
$\displaystyle p_{\star}(K_{X/Y}+L)|_{Y_{1}}$ has positive curvature, 
and it extends as a singular Hermitian metric $\wt g_{X/Y}$ on the torsion free sheaf $p_{\star}(K_{X/Y}+L)$ with positive curvature.
We say that $\wt g_{X/Y}$ is the $L^2$ metric on $p_{\star}(K_{X/Y}+L)$ with respect to $h_L$.
\end{theorem}
Since $\displaystyle p_{\star}(K_{X/Y}+L)$ is torsion free, it is locally free outside a set of codimension at least two, say $\Sigma\subset Y$.
Theorem \ref{ext} shows in particular that the restriction 
$\displaystyle \big(p_{\star}(K_{X/Y}+L)|_{Y\setminus \Sigma}, \wt g_{X/Y}\big)$ is a positively curved 
singular Hermitian vector bundle. 
\medskip

\noindent We consider next the line bundle $\det p_{\star}(K_{X/Y}+L)$; by this notation we mean the top exterior power of the direct image $p_{\star}(K_{X/Y}+L)$. Then we have the following 
consequence of the previous results.

\begin{corollary}\label{det} 
We assume that the hypothesis of Theorem {\rm \ref{ext}} are fulfilled, and that 
$\displaystyle p_{\star}(K_{X/Y}+L)$ is non-trivial.
Then the determinant line bundle 
$$\displaystyle \det p_{\star}(K_{X/Y}+L)$$ admits a singular hermitian metric whose 
curvature current $\Theta$ is positive. Moreover, we have the following statements.
\begin{enumerate}

\item[(a)] If $\Theta$ is a non-singular (1,1) form on some open subset $\Omega\subset Y\setminus \Sigma$, then the curvature
current of $\displaystyle p_{\star}(K_{X/Y}+L)|_\Omega$ is well-defined.
\smallskip

\item[(b)] If $\Theta$ vanishes on an open subset $\Omega^\prime\subset Y\setminus \Sigma$, then so does the full curvature tensor corresponding to $\displaystyle p_{\star}(K_{X/Y}+L)$. In this case the relative metric 
$\displaystyle \wt g_{X/Y}\vert_{\Omega^\prime}$ is smooth.
\end{enumerate}

\end{corollary}

\begin{proof} 
The metric $\wt g_{X/Y}$ on the direct image induces a metric on the determinant bundle 
$\displaystyle \det p_{\star}(K_{X/Y}+L)|_{Y\setminus \Sigma}$
whose curvature is positive (and it actually equals the trace of the curvature of the direct image in the complement of an algebraic set). 
It is well-known that psh functions extend across sets of codimension at least two (cf. Remark \ref{extensionremark}), hence the the first part of the 
corollary follows. 
\smallskip

The statement (a) is a direct consequence of Theorem \ref{singsing}, 
because the metric induced on the 
determinant bundle on $\Omega$ is \emph{smooth}, by standard regularity results. 

\noindent As for part (b), we use 
Theorem \ref{singsing} again, and it implies that the restriction of the curvature current corresponding to
$\displaystyle p_{\star}(K_{X/Y}+L)|_{\Omega^\prime}$ is well-defined.
We establish its vanishing next; as we will see, it is a consequence of 
the positivity in the sense of Griffiths of the curvature of $\displaystyle p_{\star}(K_{X/Y}+L)$, 
combined with the fact that its trace $\Theta$ is equal to zero.
We remark at this point that is really important to have at our disposal the {curvature current} 
as given by Theorem \ref{singsing}, and not only the positivity in the sense of \eqref{eq30}.

A by-product of the proof of Theorem \ref{singsing} (cf. \cite[Remark 4.1]{Raufi}) 
is the fact that the curvature current $\Theta$ of $\det E$ is simply the trace of the matrix-valued 
current $\displaystyle \Theta_{h_E}(E)$. By using the notations \eqref{singcurv} at the beginning 
of this section, this is equivalent to the fact that
\begin{equation}\label{mp1}
 \sum_{j, k}\sum_\alpha\mu_{j\ol k\alpha\ol \alpha}dz^j\wedge dz^{\ol k}= 0.
\end{equation}
Since $\displaystyle \Theta_{h_E}(E)$ is assumed to be positive in 
the sense of Griffiths, we infer that the current  
\begin{equation}\label{mp2}
 \sum_{j, k}\mu_{j\ol k\alpha\ol \alpha}dz^j\wedge dz^{\ol k}
\end{equation}
is positive \emph{for each index $\alpha$}. When combined with \eqref{mp1}, this implies that
\begin{equation}\label{mp3}
\mu_{j\ol k\alpha\ol \alpha}\equiv 0
\end{equation} 
for each $j, k, \alpha$. But then we are done, since the positivity of 
$\displaystyle \Theta_{h_E}(E)$ together with \eqref{mp3} shows that for each pair of indexes $\alpha, \beta$
we have \begin{equation}\label{mp4}
\Re\Big(\xi^\alpha\xi^{\ol \beta}\sum_{j, k}\mu_{j\ol k\alpha\ol \beta}v^j v^{\ol k}\Big)\geq 0
\end{equation} 
(notations as in \eqref{measure}) which in turn implies that $\displaystyle \mu_{j\ol k\alpha\ol \beta}\equiv 0$
for any $j, k, \alpha, \beta$. The current $\displaystyle \Theta_{h_E}(E)$ is therefore identically zero.
\smallskip

\noindent The regularity statement is verified as follows. In the first place we already know that the coefficients of $h$
are bounded, where $h$ stands for the local expression of the metric $\wt g_{X/Y}$. This follows thanks to relation \eqref{eq42} which implies that the absolute value of the coefficients of the dual metric $h^\star$ is bounded from above, combined with the fact that 
the determinant $\det h$ is smooth. 

Since $\dbar$ of the connection form (= curvature current) is equal to zero,
it follows that the connection is smooth. Locally near a point of $\Omega^\prime$ we therefore have
\begin{equation}\label{eq4011}
\partial h= h\cdot \Psi
\end{equation}
where $\Psi$ is smooth. The relation \eqref{eq4011}
holds in the sense of distributions; by applying the $\dbar$ operator to it, we see 
that $h$ satisfies an elliptic equation. In conclusion, it is smooth. \end{proof}

\section{Some technical results}

\noindent We recall here a result due to E. Viehweg
which has been widely used in the previous works concerning the Iitaka conjecture...

\begin{proposition}\cite[Lemma 7.3]{Vie81}\label{vi1}
Let $p:X\to Y$ be a surjective map between two non-singular, projective manifolds. 
Then there exists a commutative diagram
\begin{equation*} 
\begin{CD}
    X^\prime @>{\pi_X}>> X \\
    @Vp^\prime VV      @VVpV  \\ 
    Y^\prime @>>{\pi_Y}>   Y
\end{CD}
\end{equation*} 
such that $X^\prime$ and $Y^\prime$ are smooth, the morphisms $\pi_X, \pi_Y$ are birational,
and moreover, each hypersurface $W\subset X^\prime$ such that $\Cod_{Y^\prime}p^\prime(W)\geq 2$
is $\pi_X$-contractible, i.e., $\Cod_{X} \pi_X (W)\geq 2$.
\end{proposition}
As we see in \cite[Lemma 7.3]{Vie81}, the statement above is a quick consequence of Hironaka's flattening theorem. 

\medskip

\begin{remark}\label{natural}
Let $\Delta$ be an effective klt $\bQ$-divisor on $X$. Then we have
$$\pi_X ^\star (K_X +\Delta) +E^\prime = K_{X^\prime} + \Delta^\prime$$
where $E^\prime$ is effective and $\pi_X$-exceptional, $\Delta^\prime$ is klt.
As a consequence, there exists a set $Z\subset Y$ such that $\Cod_Y Z\geq 2$ and such that we have an isometry
\begin{equation}
p^\prime_\star(mK_{X^\prime/Y^\prime} +m \Delta^\prime)\vert_{Y^\prime\setminus Z^\prime}\to \pi_Y^\star p_\star(mK_{X/Y}+m\Delta)|_{Y\setminus Z},
\end{equation} 
where $Z^\prime:= \pi_Y^{-1}(Z)$. 
Indeed, this is simply a consequence of the definition of the 
canonical $L^2$ metrics $\widetilde{g}_{X' /Y'}$ and $\widetilde{g}_{X/Y}$ on $p^\prime_\star(mK_{X^\prime/Y^\prime} +m \Delta')$ and $p_\star(mK_{X/Y}+m\Delta)$ respectively
(cf. the constructions in the beginning of Section 4), together with the properties of the maps/manifolds in Proposition \ref{vi1}.  

\noindent Moreover, by considering a further modification, 
we can assume that the singular locus $\Sigma$ of the fibration $p'$ is normal crossing and $p^{\prime -1} (\Sigma)$ is also normal crossing.
\end{remark}
\medskip

\medskip

\noindent We will recall now two results whose combination will reveal the strategy of our proof.
\begin{theorem}\cite[Chapter 3.3]{BL}, \cite{Cao15}\label{torus} Let $T = \bC^n/\Gamma$ be a complex torus of dimension $m$, 
and let $\alpha \in H^{1,1}(T, \bZ)$ be a pseudo-effective non trivial class. If $\alpha$ is not ample
then there exists a submersion 
\begin{equation}\label{equa0120}
\pi:T\to S
\end{equation}
to an abelian variety $S$ of dimension smaller than $m$ so that we have $\alpha = \pi^\star c_1(H)$ 
for some ample line bundle $H$ on $S$. 
Moreover, after passing to some finite \'{e}tale cover, the fiber of $\pi$ is also a torus.
\end{theorem}

\begin{proof}[Sketch of the proof]
Since the pseudo-effective class on the torus can be represented by a closed constant semipositive $(1,1)$-form, 
there exists a holomorphic line bundle $L$ on $T$ with a smooth hermitian
metric $h$ such that $c_1 (L) = \alpha \in H^{1,1}(T, \bZ)$ and $i\Theta_h (L) \geq 0$.
Since $L$ is not ample, by applying \cite[Chapter 3.3]{BL}, $L$ is semiample and defines a nontrivial fibration to a subvariety of the dual torus of $T$. 
Combining with the fact that $\kappa (T)=0$, after a finite \'{e}tale cover, the image of the fibration should be an abelian variety. 
The theorem is thus proved. 
\end{proof}

\medskip

\noindent The following statement originates in the seminal work of E. Viehweg, cf. Chapter~6 of \cite{Vie95} 
as well as Proposition 4.5 in \cite{Vie01}; the generalization
presented below is stated in the article by H.~Tsuji, \cite[Section 2.6]{Tsu10}. We will nevertheless 
provide a complete treatment here, for the sake of completeness. Also, we stress 
that in the next theorem the base $Y$ 
is not necessarily the modification of an abelian variety.

\begin{theorem}\label{positivelemma}
Let $f: X\rightarrow Y$ be a fibration between two projective manifolds.
Let $L$ be a $\bQ$-line bundle on $X$ endowed with metric $h_L$ whose corresponding curvature current is 
semi-positive definite and such that $e^{-\varphi_L}$ is $L^1$-integrable on $X$, where $\varphi_L$ is the potential of $h_L$.
Let $m\in \bN$ be a positive integer such that $m L$ is a line bundle, and let $\Sigma\subset Y$ be the singular locus of $f$. We 
assume that $\Sigma$ is snc and $f^{-1} (\Sigma)$ is a normal crossing divisor. Then
there exist a constant $\varepsilon_0 >0$ and an effective $\bQ$-divisor $F$ in $X$ satisfying
$\codim_Y f_\star (F) \geq 2$, such that
\begin{equation}\label{eq3} 
c_1 (K_{X/Y} + F + L)\geq \varepsilon_0 \cdot f^\star c_1 \big(\det f_\star (m K_{X/Y}+ m L)\big).
\end{equation}
\end{theorem}

\medskip

\begin{remark}
In order to highlight the main ideas of the proof of Theorem \ref{positivelemma}, 
we first consider the following simplified case. We assume first that the fibration $f$ is smooth and $L=\mathcal{O}_X$.
In this case, there exists a natural morphism
\begin{equation}\label{naturalmorhi0case}
\det f_\star (m K_{X/Y} ) \rightarrow  \bigotimes^r f_\star (m K_{X/Y}) \qquad\text{on }Y  ,
\end{equation}
where $r$ is the rank of $f_\star (m K_{X/Y})$. We consider the fibered product 
\begin{equation}
X^r :=X\times_Y X \times_Y\cdots \times_Y X
\end{equation} 
corresponding to the map $f$.
Let $f^r : X^r \rightarrow Y$ be the natural induced fibration, and let
$\pr_i : X^r \rightarrow X$ be the projection on the $i$-th factor. We have
\begin{equation}\label{product0case}
K_{X^r /Y} =\bigotimes_{i=1}^r \pr_i ^\star (K_{X/Y})  \qquad\text{on } X^r. 
\end{equation}
As a consequence, by \eqref{naturalmorhi0case} we infer that we have a non zero section
\begin{equation}
s\in H^0 (X^r,  m K_{X^r /Y} - f^{r\star}\det f_\star (m K_{X/Y})) . 
\end{equation}
Let $\epsilon > 0$ small enough (which depends on $m$) such that if we define
$$\Delta =\epsilon \Div (s)$$ 
then the pair $\displaystyle \big(X^r, \Delta )$ is klt.

By the results in \cite{BP10}, there exists a 
very ample line bundle $A_Y\to Y$ such that for 
a generic point $y\in Y$, the restriction map 
\begin{equation}
H^0 \big(X^r , f^{r \star} A_Y + k K_{X^r  /Y}  +  k \Delta) \twoheadrightarrow 
H^0 \big(X^r _y ,  k K_{X^r /Y} + k \Delta) 
\end{equation}
is surjective for every $k$ sufficiently divisible. 
By restricting on the diagonal of $X^r$ and $X^r _y$ respectively, we have the maps
$$H^0 (X^r, f^{r \star} A_Y + k K_{X^r  /Y}  +  k \Delta) \rightarrow 
H^0 (X, r f^\star A_Y + k r(1+\epsilon m) K_{X/Y}  - kr\epsilon f^\star \det f_\star (m K_{X/Y})) ,$$
and 
$$H^0 (X^r _y , k K_{X^r /Y} + k \Delta) \twoheadrightarrow H^0 (X_y , k r(1+\epsilon m) K_{X/Y} ),$$
where the surjectivity above is a consequence of \cite{BCHM}.
In conclusion, we have the commutative diagram

\xymatrix{
H^0 \big(X^r , f^{r\star} A_Y + k K_{X^r  /Y}  +  k \Delta)  \ar[d] \ar@{->>}[r] & H^0 \big(X^r _y , k K_{X^r /Y} + k \Delta) \ar@{->>}[d]\\
H^0 (X, r f^\star A_Y +k r(1+\epsilon m) K_{X/Y} - kr \epsilon f^\star \det f_\star (m K_{X/Y}) ) \ar[r]
& H^0 (X_y , k r(1+\epsilon m) K_{X/Y} ).}

Therefore $r f^\star A_Y +k r(1+\epsilon m) K_{X/Y} - kr \epsilon f^\star \det f_\star (m K_{X/Y}) $ is effective for every $k$ sufficiently divisible.
By letting $k\rightarrow +\infty$, we obtain that 
$$ (1+\epsilon m)K_{X/Y} - \epsilon f^* \det f_\star (m K_{X/Y})$$ 
is pseudo-effective and the theorem is proved.
\end{remark}

\noindent We give now a complete proof of the theorem \ref{positivelemma}.

\begin{proof}[Proof of Thm \ref{positivelemma}]

Let $Y_0\subset Y$ be the maximal Zariski open set such that $\displaystyle f|_{f^{-1}(Y_0)}$ is flat and such that that the direct image $f_\star ( m K_{X/Y} +m L)$ is locally free when restricted to $Y_0$.
Then $\codim_{Y} (Y\setminus Y_0) \geq 2$, as is well-known (given that $X$ and $Y$ are non-singular).

\noindent By hypothesis, the inverse image of the discriminant of $f$ can be written as
\begin{equation}\label{eq4} 
f^\star \Sigma = \sum W_i + \sum a_i V_i,
\end{equation}
where $\sum W_i+ \sum V_i$ are snc and $a_i \geq 2$. Set $W := \sum W_i$ and $V:=\sum V_i$. 

Next, we see that there exists a natural morphism
\begin{equation}\label{naturalmorhi}
\det f_\star (m K_{X/Y} +m L) \rightarrow  \bigotimes^r f_\star (m K_{X/Y} +m L) \qquad\text{on }Y_0 ,
\end{equation}
where $r$ is the rank of $f_\star (m K_{X/Y})$.
In order to give a useful interpretation of \eqref{naturalmorhi}, we consider the fibered product 
\begin{equation}\label{eq5} 
X^r :=X\times_Y X \times_Y\cdots \times_Y X
\end{equation} 
corresponding to the map $f$ (as always, the convention here is that $X^r$ is the component of the 
fibered product \eqref{eq5} which maps surjectively onto $Y$).

Let $f^r : X^r \rightarrow Y$ be the natural induced fibration, and let
$\pr_i : X^r \rightarrow X$ be the projection on the $i$-th factor.
Set $X^r _0 := (f^r)^{-1} (Y_0)$ and $L_r := \bigotimes_{i=1}^r \pr_i ^\star (L)$.
According to the definition in \cite[Def 5.13]{Hor10}, the map $f$ is a flat Cohen-Macaulay fibration over $Y_0$. Moreover, the results established in 
\cite[Cor 5.24]{Hor10} show that we have the crucial equality 
\begin{equation}\label{product}
\omega_{X^r /Y} =\bigotimes_{i=1}^r \pr_i ^\star (K_{X/Y})  \qquad\text{on } X^r _0 . 
\end{equation}
Combining \eqref{product} with \cite[Lemma 3.17]{Hor10}, we infer that 
$$\bigotimes^r f_\star (m K_{X/Y} +m L)  \simeq (f^r) _\star ((\omega_{X^r /Y}\otimes L_r) ^{\otimes m}) \qquad\text{on }Y_0 .$$
As a consequence, by \eqref{naturalmorhi} we infer that we have
\begin{equation}\label{effe}
H^0 (X^r _0, (\omega_{X^r /Y} \otimes L_r)^{\otimes m} \otimes (f^{r\star}\det f_\star (m K_{X/Y} +m L))^{*} ) \neq 0 . 
\end{equation}

Let $\pi : X^{(r)} \rightarrow X^r$ be a desingularization of $X^r$ which is an isomorphism 
at non-singular points of $X^r$, and let $f^{(r)} := f^r\circ \pi$ be the map induced by $f^r$.
Set $X^{(r)} _0 := \pi^{-1} (X^r _0)$.

$$\xymatrix{
X^{(r)} \ar[d]_{\pi} \ar@/_3pc/[dd]_{f^{(r)}}\\
X^r \ar[r]^{\pr_i} \ar[d]_{f^r} &
X \ar[ld]^f \\
Y}
$$

By hypothesis $W$ is snc, so the variety $X^r _0$ is normal at each point of the subset 
$\displaystyle W_0^r := W \times_{Y_0} \cdots \times_{Y_0} W$;  moreover, it is Gorenstein with at most rational singularities. This is yet another consequence of 
\cite[Lemma 3.13, Thm 5.12]{Hor10}. 
\smallskip

\noindent We consider the canonical bundle $\displaystyle K_{X^{(r)}}$ of the manifold $X^{(r)}$; we will 
compare next its $\pi$-direct image with $\omega_{X^r}$. To this end we recall that there exists a non-zero 
morphism
\begin{equation}\label{comparison}
 \pi_\star\cO(K_{X^{(r)}})\to \omega_{X^r} 
\end{equation}
on $X^r$, cf. \cite[3.20]{Hor10}, which is moreover an isomorphism on the locus where $X^r _0$ is normal 
and has at most rational singularities.

The map \eqref{comparison} and the identity \eqref{product} show in particular the existence of a meromorphic section of the bundle
\begin{equation}\label{merom}
K_{X^{(r)}/Y}^{-1}\otimes \pi^\star\big(\bigotimes_{i=1}^r \pr_i ^\star (K_{X/Y})\big)
\end{equation}
whose zeroes and poles are contained in $X^{(r)}\setminus X^r _0$, together with the complement of the locus where
$X^r _0$ is normal 
and has at most rational singularities.

As a consequence, there exists a couple of effective divisors $E_1, E_2$ on $X^{(r)} $ such that we have
\begin{equation}\label{eq6}
K_{X^{(r)} /Y } +E_1 =\pi^\star \omega_{X^r /Y}+ E_2  \qquad\text{on }X^{(r)} _0 
\end{equation}
and moreover, each component $\Lambda$ of the support of $E_i$ belong to 
one of the following category.
\begin{enumerate}

\item[(a)] The $f^{(r)}$-image of the divisor $\Lambda$ is contained in $Y\setminus Y_0$, i.e. a set of codimension at least two.

\item[(b)] The codimension of the projection of $\Lambda$ by some maps 
$\pr_i\circ \pi$ is at least $2$, or it is equal to one of the $V_l$.
\end{enumerate}
These properties will be important for the rest of our proof.




\medskip

\noindent When combined with \eqref{effe} and \cite[III, Lemma 5.10]{Nak04}, the equality \eqref{eq6} shows that the bundle
\begin{equation}\label{eq7}
m K_{X^{(r)} /Y} + m \pi^\star L_r + m E_1 - f^{(r)\star} \det f_\star (m K_{X/Y}+ m L) + E_3
\end{equation} is effective, where $E_3$ is an effective divisor on $X^{(r)}$ which projects in codimension two, i.e. we have $(\pi\circ f^{(r)})_\star (E_3) \subset (Y\setminus Y_0)$.

Let $\epsilon > 0$ small enough (which depends on $m$ and $L_r$) such that if we define
$$\Delta =\epsilon (m K_{X^{(r)} /Y} +m \pi^\star L_r + m E_1 +E_3 - f^{(r)\star}\det f_\star (m K_{X/Y} +m L))$$ 
then the pair $\displaystyle \big(X^{(r)}_y, \Delta + \pi^\star L_r|_{X^{(r)}_y}\big)$ is klt for any $y\in Y$ 
in the complement of a set of measure zero.
Here, ``klt'' means that $e^{-2\ln |\Delta| -\sum_{i=1}^r (\pi\circ\pr_i) ^\star \varphi_L}$ is $L^1$-integrable 
on the fiber $X^{(r)}_y$. We set $\widetilde{\Delta} := \Delta + \pi^\star L_r$.

By the results in \cite{BP10}, there exists a 
very ample line bundle $A_Y\to Y$ such that for 
any point $y\in Y$ as above the restriction map 
\begin{equation}\label{surjadd}
H^0 \big(X^{(r)} , f^{(r) \star} A_Y + k K_{X^{(r)} /Y}  +  k \widetilde{\Delta}\big) \twoheadrightarrow 
H^0 \big(X^{(r)} _y , f^{(r) \star} A_Y + k K_{X^{(r)} /Y} + k \widetilde{\Delta}|_{X^{(r)} _y}\big) 
\end{equation}
is surjective for every $k$ sufficiently divisible. 

To simplify the notations 
we set $D_k = A_Y - k \epsilon\det f_\star (m K_{X/Y} +m L)$ and we have
\begin{equation}\label{addeq}
f^{(r)\star} A_Y + k K_{X^{(r)} /Y} + k \widetilde{\Delta} = f^{(r)\star} D_k  + k (1+ \epsilon m) (K_{X^{(r)}/Y}+\pi^\star L_r) + 
\epsilon k m E_1 +\epsilon k E_3. 
\end{equation}
Therefore the map \eqref{surjadd} becomes
\begin{equation}\label{surj1}
 H^0 (X^{(r)} , f^{(r)\star} D_k  + k (1+ \epsilon m) (K_{X^{(r)}/Y}+ \pi^\star L_r) + \epsilon k m E_1 +\epsilon k E_3) \twoheadrightarrow H^0 ( X^{(r)} _y, k K_{X^{(r)} /Y} + k \widetilde{\Delta}).
\end{equation}

\smallskip

\noindent As a consequence, we have the following crucial extension property.

\begin{claim}\label{r-ext}
There exists a constant $C> 0$ independent of $k$ such that for any section $u$ of the bundle 
 $\displaystyle k (1 +  \epsilon m) (K_{X/Y}+L)\vert_{X_y}$ there exists a section
\begin{equation}\label{importantrest}
U\in  H^0 (X , rf^\star D_k + r k (1 +\epsilon m) ( K_{X/Y}+L) + C k [V] + kF)
\end{equation}
whose restriction to the fiber $X_y$ is equal to $u^{\otimes r}$, where $F$ is an effective divisor on $X$ (independent of $k$) such that 
$\codim_Y f_\star (F) \geq 2$.
\end{claim}
\smallskip

\noindent We admit this statement for the time being, and we finish next the proof of Theorem \ref{positivelemma}.
By Claim \ref{r-ext} the bundle  
$$r f^\star D_k + r k (1 +\epsilon m) (K_{X/Y}+L) + C k [V] + kF\geq 0$$
is effective,
which is equivalent to say that
$$ r f^\star A_Y + r k (1 +\epsilon m) (K_{X/Y}+L) +C k [V] - r k \epsilon f^\star \det f_\star (m K_{X/Y}+mL) + kF$$
is effective as well. 
Thanks to Theorem \ref{nonreduce}, we have
$$K_{X/Y} +L \geq \varepsilon_0[V_h]$$
for some $\varepsilon_0 >0$, where $V_h$ is the divisor corresponding to the components of $V$
projecting in codimension one.
Therefore we have  
$$ r f^\star A_Y + (r k (1 +\epsilon m) + \frac{C k}{\varepsilon_0}) (K_{X/Y}+L) 
- r k \epsilon f^\star \det f_\star (m K_{X/Y}+m L) + k\widetilde{F}\geq 0 $$
for some effective divisor $\widetilde{F}$ satisfying $\codim_Y f_\star (\widetilde{F}) \geq 2$.
Theorem \ref{positivelemma} is thus proved by letting $k\to\infty$.
\end{proof}
\medskip

\noindent We establish next the claim.


\begin{proof}
The point $y\in Y$ is supposed to be generic, so we have the equality
\begin{equation}\label{eq20}
X^{(r)}_y= X_y\times\dots\times X_y
\end{equation}
where the number of the factors in the product \eqref{eq20} is $r$. Since $X_y$ is a smooth fiber, we have 
$$k(K_{X^{(r)}/Y}+ \widetilde{\Delta})+ f^{(r)\star} A_Y\vert_{X^{(r)}_y}  = k(1+\epsilon m) (K_{X^{(r)}/Y}+\pi^\star L_r) \vert_{X^{(r)}_y}$$
by construction. Thus the section $u$ we are given 
by hypothesis defines a section $u^{(r)}$ of the bundle 
\begin{equation}\label{eq21}
k(K_{X^{(r)}/Y}+ \widetilde{\Delta})+ f^{(r)\star} A_Y\vert_{X^{(r)}_y} 
\end{equation}
by considering the tensor product of the $\pr_i$-inverse images of $u$. 

\smallskip

The property \eqref{surj1} and the relation \eqref{addeq} show that there exists a section $U^{(r)}$ of the bundle
\begin{equation}\label{eq22}
f^{(r)\star} D_k + k(1+\epsilon m)(K_{X^{(r)}/Y} +\pi^\star L_r)+ \epsilon km E_1+\epsilon k E_3 , 
\end{equation}
extending $u^{(r)}$ and we show next that the ``restriction to the diagonal'' of $U^{(r)}$ satisfies all the 
properties required by the claim.
\smallskip

\noindent We recall that by properties (a) and (b), for any component $\Lambda$ of the divisor $\displaystyle E_i|_{X^{(r)} _0}$ there exists a projection 
$\pr_i\circ \pi$ for which the image of $\Lambda$ is contained in one of the $V_l$. In particular, we have
\begin{equation}\label{eq23}
E_1 +E_2 \leq C\sum_{i, l} (\pr_i \circ\pi)^\star V_l
\end{equation}
on the open set of $X^{(r)} _0$ whose complement projects in codimension greater than
2 via $\pr_i\circ\pi$. 

We consider next a covering of $X^{(r)}$ with coordinate open sets $\Omega_\alpha$ on which all 
our global bundles become trivial. Let $s_{E_i}$ be the canonical section corresponding to $E_i$.
The expression 
\begin{equation}\label{eq24}
\frac{s_{E_1} ^k }{s_{E_2}^{k+ \epsilon km}} \cdot U^{(r)}
\end{equation}
can be seen as a meromorphic section
of the bundle 
$$k(1+\epsilon m)\pi^\star\big(\bigotimes_{i=1}^r \pr_i^\star(K_{X/Y}+L)\big)+ f^{(r)\star} D_k$$
(via \eqref{eq6}) which is holomorphic on $f^{(r)}$-inverse image of the intersection of 
$Y_0$ with the set of regular values of $p$.
It corresponds to a collection of 
local meromorphic functions $(U^{(r)}_\alpha)$ for which the pole order is given by \eqref{eq24}.
 
Let $Y\setminus \Sigma$ be the set of regular values of $f$; then the fiber 
$\displaystyle X^{(r)}_{y_0}$ is equal to $\displaystyle  X_{y_0}\times\dots\times X_{y_0}$
for any $y_0\in Y\setminus \Sigma$. 
We consider the collection of \emph{holomorphic} functions 
\begin{equation}\label{eq25}
s_{V, \alpha}^{ck}U^{(r)}_\alpha\big(x,\dots, x, f(x)\big)
\end{equation}
which are defined on the diagonal subset of the product 
$\displaystyle  X_{y_0}\times\dots\times X_{y_0}$ intersected with 
$\Omega_\alpha\cap (f^{(r)})^{-1}(Y\setminus \Sigma)$. Here $s_{V, \alpha}$ denote the local equations 
corresponding to the divisor $V$, and $c$ is a constant which can be easily computed from \eqref{eq24} 
(we use the bound \eqref{eq23}).
The local holomorphic functions \eqref{eq25} glue together as a section $U$ of 
\begin{equation}\label{eq27}
rk(1+\epsilon m)(K_{X/Y}+L) + rf^{\star} D_k+ ckV
\end{equation}
except that it is only defined on some open set $X_0\subset X$ whose codimension in 
$X\setminus V$ is at least two. However, the relations \eqref{eq23} and 
\eqref{eq24} show that $U$ is bounded near the generic point of the support of $V  \cap f^{-1} (Y_0)$, so it extends to $f^{-1}(Y_0)$
and another application of \cite[III, Lemma 5.10]{Nak04} will end the proof, as follows. 

Indeed, the lemma in question shows that there exists a divisor $F$ on $X$, such that $f_\star(F)\subset Y\setminus Y_0$ and such that 
\begin{equation}\label{eq28}
f_\star\big(rk(1+\epsilon m)(K_{X/Y}+L) + ckV\big)^{\star\star}= f_\star\big(rk(1+\epsilon m)(K_{X/Y}+L) + ckV+ kF\big)
\end{equation}
where we denote by $\cF^{\star\star}$ the double dual of $\cF$. The sheaves we are dealing with are torsion free, and in this case \eqref{eq28}
combined with the projection formula show that we have
\begin{equation}\label{eq29}
f_\star\big(rk(1+\epsilon m)(K_{X/Y}+L)+ rf^{\star} D_k+ ckV\big)^{\star\star}= 
f_\star\big(rk(1+\epsilon m)(K_{X/Y}+L)+ rf^{\star} D_k+ ckV+ kF\big)
\end{equation}
(we are using here that $(\cF\otimes L)^{\star\star}= \cF^{\star\star}\otimes L$ for any locally free $L$ and torsion free $\cF$, respectively).
The claim is therefore established. 
\end{proof}

\begin{remark}\label{anyline}
The proof just finished shows that given any line bundle $P$ such that the sheaf
$$(\bigotimes^r f_\star(mK_{X/Y}+mL))\otimes P^{-1}$$
has a global (non-identically zero) section, then $\displaystyle c_1(K_{X/Y}+ L+ F)\geq \varepsilon_0 c_1(P)$.
However, the bundle $\det f_\star(mK_{X/Y}+m L)$ seems to be ``the best'' for what we have to do next.
\end{remark}
\bigskip

\noindent The following result is classical, cf. \cite{Kawa82} \cite{CH11}: it shows that in order to prove \eqref{neweqfin} 
it would be enough to establish the inequality $\kappa (K_X+\Delta) \geq \min \{1, \kappa (K_F +\Delta_F)\}$.

\begin{proposition}\label{simplecmn}
Let $p: X\rightarrow A$ be a fibration from a projective manifold to a simple Abelian variety $A$
(i.e. there is no strict subtorus in $A$) and let $\Delta$ be a klt $\bQ$-divisor on $X$. 
Let $F$ be a generic fiber of $p$.
If $\kappa (X+\Delta) \geq 1$, 
then 
$$\kappa (K_X +\Delta) \geq \kappa (K_F +\Delta_F) ,$$
where $K_F +\Delta_F = (K_X +\Delta)|_F$.
\end{proposition}

\begin{proof}
 We use here an approach which goes back to Y. Kawamata, \cite[page 62]{Kawa82}. Modulo desingularization, we can assume that the Iitaka fibration of $K_X+\Delta$ is a morphism
$\varphi: X\rightarrow W$. 
$$\xymatrix{
X\ar[rr]^-{\varphi} \ar[rd]_{p}
& & W \\
& A }
$$
Let $G$ be the generic fiber of $\varphi$ and set $\Delta_G := \Delta |_G$. Then $K_G +\Delta_G = (K_X +\Delta) |_G$ and we have 
\begin{equation}\label{iitakafiber}
\kappa (K_G+ \Delta_G) =0 .
\end{equation}
Let $p: G \rightarrow p (G)$ be the restriction of $p$ on $G$. We will analyze next among three cases
which may occur. 

\medskip

\emph{Case 1: We assume that $p(G)=A$}; then we argue as follows.
Let $\widetilde{p} : G \rightarrow \widetilde{A}$ be the Stein factorisation of $p : G \rightarrow A$:
$$\xymatrix{
G\ar[rr]^-{p} \ar[rd]_{\widetilde{p}}
& & A \\
& \widetilde{A} \ar[ru]_s}
$$
After some desingularization $\widetilde{p}$, we can assume that $\widetilde{A}$ is smooth. 
There are two subcases: 

{\em Subcase 1:} The ramification locus of $s : \widetilde{A}\rightarrow A$ is of codimension $1$ in $A$. 
Let $[E]$ be the divisor corresponding to the ramification locus. 
Since $A$ is a simple torus, $[E]$ is an ample divisor on $A$. Therefore $K_{\widetilde{A}}$ is big. In this case, it is well known that 
$\kappa (K_G +\Delta_G) \geq \dim \widetilde{A} \geq 1$. We get a contradiction with \eqref{iitakafiber}.

{\em Subcase 2:} The ramification locus of $s: \widetilde{A}\rightarrow A$ is of codimension at least $2$ in $A$. As the Stein factorisation 
extend uniquely over closed analytic subsets of codimension $\geq 2$, $\widetilde{A}$ is thus an abelian variety. 

Let $t\in \widetilde{A}$ be a generic point. Let $G_t$ be the fiber of $\widetilde{p}$ over $t$
and set $\Delta_{G_t} := \Delta|_{G_t}$.
By induction, \eqref{iitakafiber} implies
\begin{equation}\label{addedfirstpar}
\kappa (K_{G_t}+ \Delta_{G_t})=0 . 
\end{equation}

We next estimate the dimension of $G$. 
Let $F$ be the fiber of $p : X\rightarrow A$ over $s(t) \in A$. 
Then $F$ is a generic fiber.
By restricting $\varphi$ on $F$, we obtain a morphism
$$\varphi_t: F \rightarrow V$$
where $V$ is a subvariety of $W$.
Let $\widetilde{V} \rightarrow V$ be the Stein factorisation of $\varphi_t$.
$$\xymatrix{
F\ar[rr]^-{\varphi_t} \ar[rd]_{\widetilde{\varphi}_t}
& & V\\
& \widetilde{V} \ar[ru] }
$$
Since $G$ is generic, we infer that the fiber of $\widetilde{p} : G \rightarrow \widetilde{A}$ over $t$ coincides with a generic fiber of $\widetilde{\varphi}_t$.
Combining this with \eqref{addedfirstpar}, then \cite[Thm 5.11]{Uen} implies that 
$$\kappa (K_F +\Delta_F) \leq \dim \widetilde{V} = \dim F -\dim G_t .$$
Therefore we have
$$\dim G_t \leq \dim F -\kappa (K_F +\Delta_F)$$
and thus we infer that
$$\dim G = \dim G_t +\dim \widetilde{A} \leq \dim F -\kappa (K_F +\Delta_F) +\dim A = \dim X -\kappa (K_F +\Delta_F) .$$

Finally, by construction of the Iitaka fibration, $\dim G = \dim X -\kappa (K_X +\Delta)$;
we obtain the inequality
$$\dim X -\kappa (K_X +\Delta) \leq \dim X -\kappa (K_F +\Delta_F),$$
and in conclusion $\kappa (K_X +\Delta) \geq \kappa (X_F +\Delta_F)$.

\medskip

\emph{Case 2: We assume that the image $p(G)$ has dimension zero.}
Since $G$ is connected, $p (G)$ is a point in $A$.
This means that we can define a map $W\to A$, which can be assumed to be regular by blowing up $W$. We have thus the commutative diagram
$$\xymatrix{
X\ar[rr]^-{\varphi} \ar[rd]_{p}
& & W\ar[ld]^q \\
& A }
$$
Set $t:= p (G)$.
Let $F$ be the fiber of $p$ over $t$. Then $F$ is a generic fiber of $p$ and $G$ is a generic fiber of 
$$\varphi: F \rightarrow \varphi (F),$$
and by \cite[Thm 5.11]{Uen} we infer that 
$$\kappa (K_F +\Delta_F) \leq \kappa (K_G +\Delta_G) + \dim \varphi (F) = \dim \varphi (F).$$
Note that $\varphi (F)$ is the fiber of $q$ over $t\in A$.
We have $\dim W = \varphi (F) + \dim A$.
Therefore $\dim W \geq  \kappa (K_F +\Delta_F) +\dim A$.
Combining this with the fact that $\varphi$ is the Iitaka fibration, we have thus 
$$\kappa (K_X +\Delta)=\dim W \geq \kappa (K_F +\Delta_F) +\dim A ,$$
and we are done.

\medskip

\emph{Case 3: The remaining case: $p (G)$ is a proper subvariety of $A$}.
Since we are assuming that $A$ is simple, by \cite[Cor 10.10]{Uen}, any desingularization of $p(G)$
is of general type. In this case, it is well known that $\kappa (K_G +\Delta_G) \geq \dim p (G)\geq 1$. We get a contradiction.
\end{proof}

\section{Proof of the main theorem}

\noindent After the preparations in the previous sections, we will 
present here the arguments for our main result. To start with, we apply Proposition \ref{vi1} for the 
map $p:X\to A$. We will use the notation  
\begin{equation}\label{eq3011}
p^\prime: X^\prime\to A^\prime
\end{equation}
for the resulting fiber space, and
we keep in mind that the projective manifold $A^\prime$ is birational to an abelian variety. 
\begin{equation*} 
\begin{CD}
    X^\prime @>{\pi_X}>> X \\
    @Vp^\prime VV      @VVpV  \\ 
    A^\prime @>>{\pi_A}>  A
\end{CD}
\end{equation*} 
The gain is that we can allow ourselves any effective 
``error divisor" in $X^\prime$ whose $p^\prime$-projection in $A^\prime$ has codimension at least two.

We will choose next a metric on the relative canonical bundle of $p^\prime$ which reflects the properties of the 
canonical algebra of a generic fiber of $p$. 
By hypothesis, we know that $\kappa(K_{X_a} +\Delta_a)\geq 0$ for any generic $a\in A$, where $\Delta_a := \Delta |_{X_a}$. 
The important result \cite{BCHM} states 
that the canonical ring of the pair $(X_a, \Delta_a)$ is finitely generated.
Let $m \gg 0 $ be a large enough positive integer, so that the singularities of the metric of 
$\displaystyle K_{X_a} +\Delta_a$ induced by the linear system $\displaystyle m K_{X_a}+ m \Delta_a$ are \emph{minimal} (i.e. 
equivalent to the singularities given by the generators of the canonical algebra). Of course, the same
integer $m$ will work for any $a\in A^\prime$ generic. 
As we have already recalled in section 2, we can construct the $m$-Bergman kernel metric $h_{X^\prime/A^\prime}$ on 
the bundle $K_{X^\prime/A^\prime}+\Delta'$ by using the fiberwise liner systems $\displaystyle m K_{X_a} +m \Delta_a$.

\noindent Our proof relies on the 
positivity and regularity properties of the direct image sheaf
\begin{equation}\label{equa0124} 
F_m^\prime:= p^\prime_\star\big(mK_{X^\prime/A^\prime} +m \Delta^\prime\big)
\end{equation}
of the relative pluricanonical bundle of $p^\prime$. 

We introduce the bundle 
\begin{equation}\label{equa0123}
L_m^\prime:= (m-1)K_{X^\prime/A^\prime} +m \Delta^\prime
\end{equation}
and we endow it with the corresponding power of the relative Bergman metric $(m-1) h_{X^\prime/A^\prime} +h_{\Delta^\prime}$, where $h_{\Delta^\prime}$ is the 
canonical singular metric on $\Delta^\prime$; then we have
\begin{equation}\label{equa0240}
 p^\prime_\star\big((K_{X^\prime/A^\prime}+ L_m^\prime)\otimes \cI(h_{X^\prime/A^\prime}^{\otimes(m-1)} \cdot h_{\Delta^\prime})\big) \subset F_m^\prime .
\end{equation}
It turns out that \eqref{equa0240} is generically isomorphic, as we verify next.

Let $a\in A^\prime$ be a generic point and 
let $u\in H^0 (X' _a , mK_{X^\prime/A^\prime} +m \Delta^\prime)$ be a holomorphic section. 
Thanks to the explicit construction of $h_{X' /A'}$, we infer that we have
\begin{equation} 
\log|u|^2\leq (m-1)\varphi_{X^\prime/A'}|_{X_a'}+ O (1)
\end{equation}
(by a slight abuse of notation). Combining this with the fact that $(X^\prime, \Delta^\prime)$ is klt, we obtain
$u\in H^0 (X' _a,(K_{X^\prime/A^\prime}+ L_m^\prime)\otimes \cI(h_{X^\prime/A^\prime}^{\otimes(m-1)} \cdot h_{\Delta^\prime})\big)$.
Therefore \eqref{equa0240} is generically isomorphic.
We refer to \cite[Cor A.2.4]{BP10a} for more details.
\smallskip

Thanks to Theorem \ref{ext}, 
the torsion free sheaf $F_m^\prime$ can be endowed with the 
\textit{canonical $L^2$ metric} $\displaystyle \widetilde{g}_{X^\prime/A^\prime}$ with 
respect to the metric on $L_m^\prime$ constructed above; it has
positive curvature. Moreover, we see that the Hermitian line bundle 
$(K_{X'/A'} + L_m ^\prime , h_{X^\prime/A^\prime}^{\otimes(m-1)} \cdot h_{\Delta^\prime})$ satisfies the hypothesis of Theorem \ref{ext}.
We can therefore apply Corollary \ref{det} to $(F_m ^\prime , \displaystyle \widetilde{g}_{X^\prime/A^\prime})$
and thus 
the determinant line bundle 
$\det F_m^\prime$
has a positive curvature current denoted by $\Xi\geq 0$ if
endowed with the metric induced by $\displaystyle \widetilde{g}_{X^\prime/A^\prime}$.

Set $F_m := p_\star (m K_{X/A}+ m \Delta)$. By using the same construction as above, we have a \textit{canonical $L^2$ metric} $\displaystyle \widetilde{g}_{X/A}$ on $F_m$.
Thanks to Remark \ref{natural}, we have an isometry
\begin{equation}\label{addisometry}
(F_m ^\prime , \widetilde{g}_{X^\prime / A^\prime}) |_{A^\prime \setminus Z^\prime} \rightarrow \pi_A ^\star ( F_m , \widetilde{g}_{X/A} ) |_{A\setminus Z} 
\end{equation}
where $\codim_A Z \geq 2$ and $Z' := \pi_A ^{-1} (Z)$.
\medskip

\noindent Theorems \ref{torus} and \ref{positivelemma} show clearly how we will 
proceed for the rest of our proof.
Roughly speaking, if $(\pi_A)_\star \Xi$ is neither ample nor trivial, then we are done by induction,
cf. Theorem \ref{torus}. If not, then we will analyze the remaining two extreme cases in the following subsections.

\subsection{The direct image of the curvature current of the determinant is non-zero}\label{case1} 
By the construction at the beginning of this section, we have a commutative 
diagram satisfying the properties in Proposition \ref{vi1} and Remark \ref{natural}:
\begin{equation*} 
\begin{CD}
    X^\prime @>{\pi_X}>> X \\
    @Vp^\prime VV      @VVpV  \\ 
    A^\prime @>>{\pi_A}>   A
\end{CD}
\end{equation*} 
If the  
curvature current corresponding to the determinant of $F_m^\prime$ is non-identically zero on $A^\prime \setminus Z^\prime$, i.e.
\begin{equation}\label{0126}
\pi_{A\star}\Xi\neq 0,
\end{equation}
then $\Xi$ is automatically a big class in $H^{1,1} (A^\prime,\bZ)$ (cf. Thm \ref{torus})
if $A$ is a simple torus\footnote{We can add this assumption by the argument in the beginning of Theorem \ref{mainourthem}.}. 
\medskip

\noindent As a consequence of Theorem \ref{positivelemma} we obtain the 
following result.

\begin{corollary}\label{impcor}
If the class $\{\Xi\}$ is big, then
for any generic point $a\in A^\prime$ and for
any $k\gg 1$ sufficiently divisible we can find an effective divisor $E$ in $X'$
satisfying $\codim_{A'} p' _\star (E) \geq 2$, such that the restriction map
$$H^0 (X^\prime, k K_{X^\prime/A^\prime} + k \Delta^\prime +E) \rightarrow H^0 (X^\prime _a , k K_{X^\prime/A^\prime} +k\Delta^\prime)$$
is surjective.
\end{corollary}

\begin{proof}
Thanks to Remark \ref{natural}, we can assume that the singular locus of $p'$ is normal crossing and we can thus apply Theorem \ref{positivelemma} to the fibration $p'$.
Let $A_Y$ be a very ample divisor on $A'$.
Since $\det p' _\star (m K_{X^\prime/A^\prime} +m \Delta^\prime) =\Xi$ is big on $A^\prime$, by Theorem \ref{positivelemma} 
we can find a parameter $m_1\in \bN$ large enough, and an effective divisor $E$ satisfying 
$\codim_{A'} p' _\star (E) \geq 2$, such that 
\begin{equation}\label{inequality}
m_1 K_{X^\prime/A^\prime} + m_1 \Delta^\prime + E= (p')^\star A_Y +\Delta_1.
\end{equation}
for some pseudo-effective divisor $\Delta_1$ on $X'$.

\noindent By \cite{BP10}, the restriction map
\begin{equation}\label{surj}
H^0 (X^\prime,  k K_{X^\prime/A^\prime} +k \Delta^\prime + (p')^\star A_Y +\Delta_1) \rightarrow H^0 (X^\prime _a , kK_{X^\prime/A^\prime} + k\Delta^\prime +
(p')^\star A_Y +\Delta_1)
\end{equation}
is surjective, for any $k\geq m_1$ divisible enough.
Indeed, this can be seen as follows: $\Delta_1$ is pseudo-effective, so its first Chern class 
contains a closed positive current  
$T$. If we choose $a\in A^\prime$ very general, and $m_1\gg 0$, then the restriction of the current
$\displaystyle \frac{1}{k}T$ to $X^\prime_y$ is well-defined, and the corresponding multiplier ideal sheaf $\mathcal{I} (\Delta^\prime + \frac{1}{k} T)$
is trivial for any $k\geq m_1$. We endow the $\bQ$-bundle 
$$K_{X^\prime/A^\prime}+ \Delta^\prime +\frac{1}{k}\Delta_1$$
with the $k$-Bergman metric, and then we write
$$k K_{X^\prime/A^\prime} +k \Delta^\prime + (p')^\star A_Y +\Delta_1= K_{X^\prime}+ \Delta^\prime+ \frac{1}{k}\Delta_1+ 
(k-1)(K_{X^\prime/A^\prime} + \Delta^\prime + \frac{1}{k}\Delta_1)+ (p')^\star A_Y$$
so that the surjectivity of the map \eqref{surj} follows by the usual extension results.
 
The relation \eqref{inequality} shows that we have 
$$k K_{X^\prime/A^\prime} + k \Delta^\prime + (p')^\star A_Y +\Delta_1= 
(m_1+ k)(K_{X^\prime/A^\prime}+\Delta^\prime) +E, $$ we therefore infer
$$H^0 (X^\prime, (m_1 +k)( K_{X^\prime/A^\prime} +\Delta^\prime )+E)\twoheadrightarrow  H^0 (X^\prime _a , (m_1 +k) (K_{X^\prime/A^\prime} +\Delta') +E)$$
and this last vector space is equal to 
$$H^0 (X^\prime _a , (m_1 +k) (K_{X^\prime/A^\prime} +\Delta^\prime)),$$
as a consequence fact that $p' _\star (E) \subsetneq A'$.
The corollary is proved.
\end{proof}


\medskip

\noindent In particular we have $\kappa (K_X +\Delta) \geq \kappa (K_{X_a} +\Delta_a)$ by Proposition \ref{vi1}.

\medskip

\subsection{The curvature current of the determinant is zero}\label{case2} 
In this subsection we assume that $p: X\rightarrow A$ is a fibration from a projective manifold
to an abelian variety, and we have
\begin{equation}\label{0125}
\pi_{A, \star}\Xi= 0 ,
\end{equation} 
where $\Xi$ is the determinant of $F_m ^\prime$ as in \eqref{0126}.

Set $F_m := p_\star (m K_{X/A}+m \Delta)$ and let $\wt g_{X/A}$ be the corresponding canonical $L^2$ metric on $F_m$ constructed in the beginning of Section 4.
By \eqref{addisometry}, it follows that the curvature of $\det (F_m)$ is equal to zero on $A\setminus \Sigma$, 
i.e. in the complement of a set of codimension at least two. However, it is known that (cf. for example 
\cite[Chapter III, Cor 2.11]{Dem}) the support of a closed 
positive $(1,1)$-current cannot be contained in 
such a small set, unless the said current is identically zero. We therefore have
\begin{equation}\label{01125}
\Theta\big(\det(F_m), \det \wt g_{X/A}\big)= 0, 
\end{equation} 
and it follows that there exists a subset of $A$ still denoted by $\Sigma$ such that the next properties hold true.
\begin{enumerate}

\item[(a)] The codimension of $\Sigma$ in $A$ is at least two, i.e. $\codim_A\Sigma\geq 2$.
\smallskip

\item[(b)] The restriction of the direct image sheaf $\displaystyle F_m|_{A\setminus \Sigma}$ 
is a vector bundle.

\item[(c)] The  
canonical $L^2$ metric $\wt g_{X/A}$ is non-singular on $A\setminus \Sigma$, and the couple
$\displaystyle \big(F_m|_{A\setminus \Sigma}, \wt g_{X/A}\big)$ is a Hermitian flat vector bundle.
\end{enumerate}
This is a consequence of Corollary \ref{det}, combined with \eqref{01125}.
In what follows, we will write 
\begin{equation}\label{eq50}
(\cE, h):= \big(F_m, \wt g_{X/A}\big)|_{A\setminus \Sigma}
\end{equation}
in order to simplify the notation. Thus $\cE\to A\setminus \Sigma$ is a vector bundle endowed with a 
non-singular metric $h$ whose associated curvature is equal to zero. Then the parallel transport induces a representation
\begin{equation}\label{paralelltran}
\rho: \pi_1(A \setminus \Sigma)\to \mathbb{U} (r) , 
\end{equation}
where $r$ is the rank of $\cE$ and where $\mathbb{U} (r)$ is the unitary group of degree $r$.
\medskip

\noindent The next proposition is a consequence of the fact that $\pi_1 (A)$ is commutative.



\begin{proposition}\label{notorsion}
Let $\rho: \pi_1(A)\to \mathbb{U} (r)$ be the representation {\rm \eqref{paralelltran}}. 
The following assertions hold true.

\begin{enumerate}

\item [(i)]
If the image of $\rho$ is infinite, then there exists a non-zero section of the bundle
$mK_X+ m\Delta+ L$, where $L$ is a non-torsion topologically trivial line bundle on $X$.
\smallskip

\item [(ii)] If the image of $\rho$ is finite, we have $\kappa (K_X +\Delta) \geq \kappa (K_F +\Delta_F)$.
\end{enumerate}
\end{proposition}

\begin{proof}
Since $\pi_1 (A) =\pi_1 (A \setminus \Sigma)$ is abelian, $\rho$ can be decomposed as the direct sum of $r$ representations $\displaystyle \{ \rho_i\}_{i=1}^r$
$$\rho_i : \pi_1 (A \setminus \Sigma) \rightarrow \mathbb{U} (1) .$$

\noindent {\em (i): } If the image of $\rho$ is infinite, there exists at least one index, say $i=1$, such that the image of the corresponding representation $\rho_1$ is infinite. Then $\rho_1$ corresponds to a
topologically trivial \emph{non-torsion} line bundle $L_1$. 
We therefore infer the existence of a section 
\begin{equation}
s\in H^0 \big(p^{-1} (A\setminus \Sigma), m K_X +m\Delta - L_1\big)
\end{equation}
such that $\displaystyle |s|_{\wt g_{X/A}} (y) =1$ for every generic point $y\in A$. 

It follows that
the $L^2$-norm of $s$ with respect to any smooth metric on the bundle $\displaystyle m K_X +m\Delta - L_1$ is finite.
Indeed, this is a consequence of the fact that the weights of the metric $h^{(m)}_{X/Y}$ are bounded from above,
cf. Theorem \ref{rel2}.  
Then $s$ can be extended as an element in $H^0 (X, m K_X +m\Delta - L_1)$ and the claim (i) of our proposition is established.
\smallskip

\noindent {\em (ii): } If the image of $\rho$ is finite, set $l :=|\rho (\pi_1 (A))|$. For every $\tau\in H^0 (F, m K_F +m\Delta)$, the paralel transport of $\tau^{\otimes l}$
induces an element in 
$$s\in H^0 (p^{-1} (A\setminus \Sigma), ml K_X +ml\Delta) ,$$ 
We invoke the same argument as above (namely the finiteness of the $L^2$ norm of $s$) to show that 
it extends to $X$. It defines therefore a non-zero element in $H^0 (X, ml K_X +ml\Delta)$, and the proposition is proved.
\end{proof}


\bigskip

\medskip


\noindent The following result due to Campana-Peternell \cite[Thm 3.1]{CP11} together with its generalization 
in \cite{CKP12} will allow us to conclude.

\begin{theorem}\label{cape}\cite[Thm 3.1]{CP11}\cite[Thm 0.1]{CKP12}
Let $X$ be a projective complex manifold, and 
let $\Delta$ be an effective $\bQ$-divisor on $X$, such that the pair $(X, \Delta)$ is log canonical. 
Let $L\in \Pic^0(X)$ be a topologically trivial line bundle. Then:
\begin{enumerate}

\item[(a)] For any positive integer $m\geq 1$ we have $\kappa(K_X +\Delta)\geq \kappa(mK_X+ m\Delta+ L)$.
\smallskip
\item[(b)] Let $l \in \mathbb{N}$ be a positive integer. 
Then there exists infinitely many $d\in\mathbb{N}$ for which we have
$$h^0 (X, l(K_X +\Delta) +d L) \geq h^0 (X, l(K_X +\Delta) +L).$$
\smallskip
\item[(c)] If $\kappa(K_X +\Delta)= \kappa(m K_X+ m\Delta + L)=0$, then $L$ is a torsion bundle.
\end{enumerate}
\end{theorem} 

\noindent In \cite{CKP12} the point $(b)$ is not explicitly stated, so  
we will provide here a complete treatment for the convenience of the readers. 
The argument we invoke in what follows is borrowed from \cite[Prop 3.2, (4)]{CP11}.

\begin{proof}
We remark that the point $(a)$ is a direct consequence of \cite[Thm 0.1]{CKP12}. 
In order to establish $(b)$, we first observe that for each $q$ and $l$, the set 
$$V_{q,l} = \{\lambda \in \Pic^0 (X) :  h^0 (X, l(K_X +\Delta) +\lambda) \geq q\}$$
\emph{is a finite union of torsion translates of complex subtori of} $\Pic^0 (X)$, 
cf. \cite[Section 2]{CKP12}.  In particular, if $L\in V_{q,l} $, then there exists a 
torsion line bundle $\rho$ and a subtorus $T\subset \Pic^0 (X)$ such that $L$ belongs to the component $\rho+ T$ of $V_{q,l}$. If we denote by $s$ the order of $\rho$, than  we see that   
$d\cdot L\in V_{q,l}$ provided that $d:= p\cdot s+ 1$, where $p$ is any positive integer.
The point $(b)$ is thus proved.
\smallskip

To prove the point $(c)$, we suppose by contradiction that there exists a non-torsion bundle $L\in \Pic^0 (X)$ such that $\kappa(K_X +\Delta)= \kappa(m K_X+ m\Delta + L)=0$.
After passing to some multiple of $m$ and $L$, we can assume for simplicity that  
$$h^0 (X, m K_X + m \Delta +L) = 1 .$$ 
As $V_{1, m}$ is a finite union of torsion translates of complex subtori of $\Pic^0 (X)$,
we can find a torsion bundle $L_{\tor}\in \Pic^0 (X)$ and a non-trivial bundle $F\in T'$,
(where $T'$ is a subtorus of $\Pic^0 (X)$), such that $L= L_{\tor} +F$ and $L_{\tor}+T' \subset V_{1,m}$.
As a consequence, for every $t\in\mathbb{R}$, we can find three non-trivial sections 
\begin{equation}\label{threesect}
s_t \in H^0 (X, m K_X + m \Delta + L_{\tor} +t F),\quad
s_{-t} \in H^0 (X, m K_X + m \Delta + L_{\tor} -tF)
\end{equation} 
as well as $s_0 \in  H^0 (X, m K_X + m \Delta + L_{\tor})$.
When $|t|$ is small enough, $s_t \cdot s_{-t}$ and $s_0 \cdot s_0$ are two linearly independent elements in $H^0 (X, 2m K_X + 2m \Delta + 2 L_{\tor})$.
Then $\kappa(K_X +\Delta) \geq 1$ and we get a contradiction.
\end{proof}

\medskip


\noindent We are now ready to prove our main theorem.

\begin{theorem}\label{mainourthem}
Let $p: X\rightarrow A$ be a fibration from a projective manifold to an Abelian variety.
Let $\Delta$ be an effective klt $\bQ$-divisor on $X$ and let $F$ be a generic fiber of $p$.
Then
\begin{equation}\label{eqfin}
\kappa (K_X +\Delta) \geq \kappa (K_F +\Delta_F) , 
\end{equation}
where $\Delta_F = \Delta|_F$.
\end{theorem}

\begin{proof}
Without loss of generality, we can assume that $A$ is a simple torus, i.e., there is no nontrivial subtorus in $A$. In fact, since $A$ is Abelian, if there is a nontrivial 
subtorus $A_1$,
by Poincare's reductibility theorem (cf. \cite[Thm 8.1]{Deb}), after a finite smooth cover, $A = A_1 \times A_2$, where $A_2$ is another subtorus of $A$.
Then \eqref{eqfin} can be proved by induction. 

We follow the notations in the beginning of Section 4. In particular, we have the commutative diagram
\begin{equation*} 
\begin{CD}
    X^\prime @>{\pi_X}>> X \\
    @Vp^\prime VV      @VVpV  \\ 
    A^\prime @>>{\pi_A}>   A
\end{CD}
\end{equation*} 
and the positive curvature current $\Xi$ on $A^\prime$.

If $(\pi_A)_\star (\Xi) \neq 0$, we can use the results in the case 4.1.
In particular, as $A$ is simple, Theorem \ref{torus} implies that $\Xi$ is big on $A'$.
By using Corollary \ref{impcor} and Proposition \ref{vi1}, we obtain $\kappa (K_X +\Delta) \geq \kappa (K_F +\Delta_F)$.

If $(\pi_A)_\star (\Xi) =0$, we are in the case 4.2. If the image of the representation $\rho$ in Proposition \ref{notorsion} is finite, 
then Proposition \ref{notorsion} implies \eqref{eqfin}.
If not, using Proposition \ref{notorsion} again, there exists $L\in \Pic^0 (A)$ which is not a torsion point, 
such that 
\begin{equation}\label{equa0203}
\kappa (m K_{X/A} + m \Delta+ p^\star L) \geq 0.
\end{equation} 
Combined this with Theorem \ref{cape}, we have 
\begin{equation}\label{equa0202}
\kappa(K_X +\Delta)\geq 1.
\end{equation} 
Indeed, we necessarily have $\kappa(K_X +\Delta)\geq 0$, and we cannot have equality as it would contradict \eqref{equa0203}
and the point (b) of Theorem \ref{cape}.
Thanks to \eqref{equa0202} and Proposition \ref{simplecmn}, we obtain \eqref{eqfin} and the theorem is proved.
\end{proof}


\section{Further results and remarks }

\noindent Throughout the current section, the notations we will be 
using are the same as in section 4.  The following result is a direct consequence of the arguments used in paragraph 4.1, so we simply state it without any comment about the proof.

\begin{theorem}\label{conseqbig} Let $p:X\to Y$ be an algebraic fiber space, and let $\Delta$ be an effective $\bQ$-divisor on $X$ such that $(X, \Delta)$ is klt. We assume that for some positive $m$ divisible enough we have 
$\det F_m$ is big. Then we infer that
\begin{equation}\label{yesman1}
\kappa (K_X +\Delta) \geq \kappa (K_F +\Delta_F)+ \kappa(Y), 
\end{equation}
where $\Delta_F = \Delta|_F$.
\end{theorem}
\medskip

\noindent Developing further the ideas in the proof of our main result here, we obtain the following
statement in which the flatness of $\det F_m$ is shown to have stronger consequences. This result was suggested to us by Christian Schnell.

\begin{theorem}\label{flat} Let $p:X\to Y$ be a fibration between two projective manifolds and let $\Delta$ be an effective $\mathbb{Q}$-divisor, such that $(X, \Delta)$ is klt.
If $\displaystyle \det F_m$ is topologically trivial,
then $\displaystyle (F_m, \widetilde{g}_{X/Y})$ is a hermitian flat vector bundle on $Y$.
\end{theorem}

\begin{proof} We use the same argument as in the beginning of subsection \ref{case2}, and infer 
the existence of a subset of $\Sigma\subset Y$ such that the properties (a)--(c) 
cf. \ref{case2}, right after \eqref{01125},
hold true.

\noindent As we will see next, our result is a consequence of the correspondence between the 
unitary representations of the fundamental group of $Y$ and the Hermitian flat vector bundles 
on $Y$, together with the fact that $\codim_Y\Sigma\geq 2$.

Let $(U_\alpha)$ be a finite cover of $Y$ with contractible coordinate sets. For each index $\alpha$ the set $U_\alpha\setminus \Sigma$ is equally contractible --this is a consequence 
of (a) we are referring above. The parallel translation with respect to the flat metric 
$\displaystyle \wt g_{X/Y}|_{Y\setminus \Sigma}$ gives a holomorphic trivialization
\begin{equation}\label{20160208}
\theta_\alpha: F_m|_{U_\alpha\setminus \Sigma}\to \big(U_\alpha\setminus \Sigma\big)\times \bC^r
\end{equation}
which is moreover an \emph{isometry} provided that the right hand side in \eqref{20160208} 
is endowed with the Euclidean metric. Therefore the transition functions, say 
$g_{\alpha\beta}$, corresponding to 
the trivializations $\theta_\alpha$  are matrices in ${\mathbb U}(r)$.

Let $(\cE, h_E)$ be the Hermitian vector bundle of rank $r$ on $Y$ given locally by 
$\displaystyle \big(U_\alpha\times \bC^r\big)_{\alpha}$, together with the transition functions 
$\big(g_{\alpha\beta}\big)_{\alpha, \beta}$ (so that the metric $h_E$ corresponds to the flat metric on 
$U_\alpha\times \bC^r$). 
This construction gives an isometry 
\begin{equation}\label{20160209}
i: F_m|_{Y\setminus \Sigma}\to \cE|_{Y\setminus\Sigma}.
\end{equation}
By Hartogs' theorem, $i$ extends to an injection of sheaves 
\begin{equation}\label{exteni}
F_m \rightarrow \cE
\end{equation}
on $Y$ which we still denote by $i$. We show next that in fact the map \eqref{exteni} is 
an isomorphism.

To this end let $U$ be a coordinate topologically open subset of $Y$ and $u\in H^0 (U, \cE)$ be 
a holomorphic section whose norm at each point of $U$ is equal to one, i.e. 
\begin{equation}\label{20160210}
|u|_{h_E} (y) =1
\end{equation}
for each $y\in U$. By \eqref{20160209}, the restriction of $u$ to $U\setminus \Sigma$ belongs to the image of $i$, so that we have 
\begin{equation}
u|_{U\setminus \Sigma}= i(v_0)
\end{equation}
for a $v_0\in H^0 (U\setminus \Sigma, F_m)$. The main point next is that $i$ is an isometry,
so we infer that 
\begin{equation}
|v_0|_{\widetilde{g}_{X/Y}} (y) =1
\end{equation}
for any $y\in U\setminus\Sigma$. 

The section $v_0$ corresponds to an element in 
$\displaystyle H^0\big(p^{-1}(U\setminus \Sigma), m K_{X/Y}+m\Delta\big)$ whose pointwise norm with respect to the canonical $L^2$ metric is equal to one. Therefore, we are in the same situation as in Proposition \ref{notorsion}, so $v_0$ extends across the inverse image of $\Sigma$. In conclusion, we have constructed a local section $v$ of $\displaystyle F_m|_{U}$ such that 
$i(v)= u$ and the theorem is proved.
\end{proof}

\medskip

\noindent A rather immediate application of Theorem \ref{positivelemma} and Theorem \ref{flat}, 
is the next result, which already appears in \cite{konnar} in a slightly less general form.

\begin{lemma}\label{onethenall}
Let $p: X\rightarrow Y$ be an algebraic fiber space, and let $\Delta$ be an effective $\mathbb{Q}$-divisor, such that $(X, \Delta)$ is klt. 
If there exists an integer $m \geq 2$ such that
$$c_1 (p_\star (m K_{X/Y} +m \Delta))=0 \in H^{1,1} (Y, \mathbb{R}), $$
then $F_{m_1} := p_\star (m_1 K_{X/Y} +m_1 \Delta)$ is hermitian flat for any $m_1 \in \mathbb{N}$ 
such that 
$F_{m_1}$ is non zero.
\end{lemma}

\begin{proof} Let $m_1$ be a positive integer as above. The idea of the proof is very simple: the positivity carried by the determinant of $\displaystyle F_{m_1}$ can be injected into $\displaystyle F_{m}$ via Theorem \ref{positivelemma}. Since we already know that the first Chern class of $\displaystyle F_{m}$ is equal to zero, the conclusion follows. We give next the details.

As we have already seen, the determinant 
$L_1: = \det F_{m_1}$ of the sheaf $\displaystyle F_{m_1}$ is a pseudo-effective line bundle.
There exists thus a metric $h_L$ on $L_1$ such that 
$$\Theta_{h_L} (L_1) \geq 0$$
in the sense of currents on $Y$. 
By applying Theorem \ref{positivelemma} (with the $\bQ$-bundle ``$L$'' corresponding to $\Delta$)
we infer the existence of an effective divisor $T$ on $X$ such that 
the codimension of the projection is large, i.e. $\displaystyle  \codim_Y (p_\star (T)) \geq 2$ and such that $\displaystyle c_1 (K_{X/Y}+\Delta +T - \epsilon_1 p^\star L_1)$ is pseudo-effective on $X$ 
for some positive $\epsilon_1 >0$.

As a consequence, we can equip the bundle $(m-1)(K_{X/Y}+\Delta) +T$ with a metric $h$ such that 
\begin{equation}\label{basepositive}
\Theta_h \big((m-1)(K_{X/Y}+\Delta) +T\big) \geq \epsilon_2 p^\star\Theta_{h_L} (L_1)
\end{equation}
on $X$, where $\epsilon_2 >0$ is a positive real number, and such that
\begin{equation}\label{multi}
H^0 \left(X_y , \big(m(K_{X/Y}+ \Delta)+ T\big)\otimes \mathcal{I} (h|_{X_y})\right) = 
H^0 \left(X_y,  m(K_{X/Y}+ \Delta)+ T|_{X_y}\right)
\end{equation}
for generic $y\in Y$.
\smallskip

\noindent Let $h_{X/Y}$ be the metric on $p_\star (m K_{X/Y}+m\Delta +T)$ induced by $h$ as in Theorem \ref{ext}.

By Lemma \ref{base+} below, the relation \eqref{basepositive} 
gives the following lower bound for the curvature of the determinant: we have
\begin{equation}\label{rev1}
\Theta_{\det_{h_{X/Y}}} \!\!\! \det p_\star (m K_{X/Y}+m\Delta +T) \geq r \epsilon_2 \Theta_{h_L} (L_1),
\end{equation}
on $Y$, where $r$ is the rank of the direct image 
$\displaystyle p_\star (m K_{X/Y}+m\Delta +T)$.
\smallskip

Since $\codim_Y (p_\star (T)) \geq 2$, we have 
$$\det p_\star (m K_{X/Y}+m\Delta +T) = \det p_\star (m K_{X/Y}+m\Delta) = 0 \in H^{1,1} (Y, \mathbb{R}).$$
Combining this with \eqref{rev1}, we know that 
the current $\displaystyle \Theta_{h_L} (L_1)$ is identically zero, and the lemma is proved.
\end{proof}
\medskip

\noindent The following statement was used during the proof of Lemma \ref{onethenall}.

\begin{lemma}\label{base+} Let $p:X\to Y$ be an algebraic fiber space. Let 
  $(G, h_G)$ and $(L, h_L)$ be Hermitian line bundles on $X$ and $Y$, respectively. 
The metrics $h_G$ and $h_L$ are allowed to be singular and we assume that the 
following requirements hold true.
\begin{enumerate}

\item[(i)] We have $\displaystyle \Theta_{h_L}(L)\geq 0$ and 
\begin{equation}\label{onX}
\Theta_{h_G}(G)\geq \ep_0p^\star\Theta_{h_L}(L) 
\end{equation}
on $X$, where $\ep_0$ is a positive real number.
\smallskip

\item[(ii)] The direct image $p_\star(K_{X/Y}+ G)$ is non-zero, and we have
\begin{equation}\label{intcond}
p_\star\left((K_{X/Y}+ G)\otimes \cI(h_G)\right)_y= p_\star\left(K_{X/Y}+ G\right)_y
\end{equation}
for any generic $y\in Y$.
\end{enumerate}
Let $h_{X/Y}$ be the $L^2$ metric on the direct image sheaf 
$\displaystyle p_{\star}\left(K_{X/Y}+ G\right)$ induced by $h_G$.   
Then we have 
\begin{equation}\label{base++}
\Theta_{\det h_{X/Y}}\left(\det p_{\star}(K_{X/Y}+ G)\right)\geq r\ep_0\Theta_{h_L}(L) 
\end{equation} 
on $Y$.
\end{lemma}

\begin{proof}Our arguments are based on the main results in \cite{PT} combined with 
very basic properties of psh functions, as follows.

Let $y_0\in Y$ be an arbitrary point, and let $\Omega_0\subset Y$ 
be an open subset of $Y$ centered at $y_0$. Let $\Omega\subset p^{-1}(\Omega_0)$ 
be an open subset of $X$ 
contained in the inverse image of $\Omega_0$. We assume that the restriction of $h_G$ to $\Omega$ is given by the weight $e^{-\varphi_G}$, and that the restriction of $h_L$ to $\Omega_0$ is given by $e^{-\varphi_L}$. By hypothesis we have 
\begin{equation}\label{5231}
\sqrt{-1}\ddbar \varphi_L\geq 0, \quad \sqrt{-1}\ddbar \varphi_G\geq \varepsilon_0
p^\star\sqrt{-1}\ddbar \varphi_L
\end{equation}       
on $\Omega$. In particular, we can write the local weight of $h_G$ 
\begin{equation}\label{5232}
\varphi_G= (\varphi_G-\ep_0\varphi_L\circ p)+ \ep_0\varphi_L\circ p 
\end{equation}
as sum of two psh functions (actually the difference 
$\varphi_G-\ep_0\varphi_L\circ p$ is only psh up to modification in the complement of a measure zero set, but this is not relevant for what we have to do next).

We remark that the local expressions 
$\displaystyle \varphi_G-\ep_0\varphi_L\circ p$ glue together as a global metric on 
$G|_{p^{-1} (\Omega_0)}$; the resulting object is denoted by 
\begin{equation}\label{5233} 
h_{0G}:= e^{-\ep_0\varphi_L\circ p}h_G
\end{equation}
and we have $\displaystyle \Theta_{h_{0G}}(G|_{p^{-1}(\Omega_0)})\geq 0$. By \cite{PT}, 
the restriction 
\begin{equation}\label{5234}
p_\star(K_{X/Y}+ G)|_{\Omega_0}
\end{equation} 
is positively curved when endowed with the $L^2$ metric induced by $h_{0G}$. 
Here the important fact is that the hypothesis \eqref{intcond} implies a similar relation for 
$h_{0G}$.
  
Let $\xi$ be a local holomorphic section of the dual of the bundle \eqref{5234}. The 
positivity of the curvature together with the expression \eqref{5233} show that we have
\begin{equation}\label{5235}
\sqrt{-1}\ddbar \log|\xi|_{h_{X/Y}}^2\geq \ep_0\sqrt{-1}\ddbar\varphi_L
\end{equation}
and this implies \eqref{base++}.   
\end{proof}

\medskip

\noindent We recall next the following deep result due to K.\ Zuo \cite[Cor 1]{Zuo96} (see also \cite[Thm 1]{CCE}); it will play a crucial role in our next statement. 

\begin{theorem}\label{mathkang} {\rm (\cite[Cor 1]{Zuo96})} Let $Y$ be a compact K\"ahler manifold, and let $\rho$ be a finite dimensional representation of $\pi_1(Y)$, whose Zariski closure is a reductive algebraic group. If $\kappa(Y)= 0$, the $\rho$ splits as a direct sum of representations of rank one, modulo an \'etale cover of $Y$. 
\end{theorem}

\medskip

\noindent As a consequence of the considerations in this paper together with Theorem \ref{mathkang},
we obtain the following main theorem of this section.

\begin{theorem}\label{conseqflat} Let $p:X\to Y$ be an algebraic fiber space, 
and let $\Delta$ be an effective $\bQ$-divisor on $X$ such that $(X, \Delta)$ is klt. We assume that for some $m \geq 2$ the line bundle
$\det F_m$ is topologically trivial. Then we have
\begin{equation}\label{yesman2}
\kappa (K_X +\Delta) \geq  \kappa (K_F +\Delta_F) + \kappa(Y), 
\end{equation}
where $\Delta_F = \Delta|_F$.
\end{theorem}

\begin{proof}

Since $\det F_m  =0$ we infer that $F_m$ is hermitian flat vector bundle of rank $r$ over $Y$,
by Theorem \ref{flat}.
We obtain a representation 
$$\rho: \pi_1 (Y) \rightarrow \mathbb{U} (r)$$
induced by the parallel transport corresponding to $(F_m, \wt g_{X/Y})$. 

\noindent The unitary group 
${\mathbb U}(r)$ is compact, therefore the Zariski closure of the image of $\rho$ is 
automatically \emph{reductive}. 
\medskip

\noindent We are therefore lead to distinguish among three cases, corresponding to the Kodaira dimension of $Y$.
\smallskip

\noindent $\bullet$ \emph{Case 1: The Kodaira dimension of $Y$ is equal to $-\infty$, i.e. $\kappa(Y)= -\infty$.}
In this case, we have immediately \eqref{yesman2}.

\medskip

\noindent $\bullet$ \emph{Case 2: The Kodaira dimension of $Y$ is equal to zero, i.e. $\kappa(Y)= 0$.} Thanks to Theorem \ref{mathkang}, we infer that $\rho$
splits into a direct sum of $1$-dimensional representations, after a finite \'{e}tale 
cover $p_Y: Y^\prime\rightarrow Y$. We denote by $X^\prime$ the fibered product of $X$ and $Y^\prime$ 
over $Y$, $p_X : X^\prime \rightarrow X$ the natural projection and $a: Y^\prime \rightarrow \Alb Y^\prime$ the Albanese map of $Y^\prime$.
\begin{equation*} 
\begin{CD}
    X^\prime @>{p_X}>> X \\
    @Vp^\prime VV      @VVpV  \\ 
    Y^\prime @>{p_Y}>>   Y \\
     @VaVV\\
    \Alb Y^\prime \\
\end{CD}
\end{equation*} 
We have thus 
\begin{equation}\label{directsumtrivial}
p' _\star (m K_{X' /Y'} + m\Delta') = p_Y ^\star F_m =\bigoplus_{j=1}^r L_j , 
\end{equation}
where $r$ is the rank of $F_m$, $\Delta' =p_X ^\star \Delta$ and $L_j \in \Pic^0 (Y')$. 

Thanks to Lemma \ref{onethenall}, we can assume that $m$ is large and sufficiently divisible. 
Then $\kappa (Y)=0$ implies that $a_\star (m K_{Y'})$ is a trivial line bundle.
Since $L_1 ,\cdots , L_r$ are pulled back from $\Alb Y'$, we have $L_j = a^\star \widetilde{L}_j$ for some $\widetilde{L}_j \in \Pic^0 (\Alb Y')$.
We have thus 
\begin{equation}\label{decomp}
(a\circ p')_\star (m K_{X'} + m\Delta') \simeq a_\star (m K_{Y'}) \otimes (\bigoplus_{j=1}^r \widetilde{L}_j ) \simeq \bigoplus_{j=1}^r \widetilde{L}_j . 
\end{equation}
Set $V_{1, m}:=\{L\in \Pic^0 (\Alb Y') | h^0 (m K_{X'} + m\Delta' + L) \geq 1 \} $.
Then \eqref{decomp} implies that 
\begin{equation}\label{adfinitset}
 V_{1, m} = \{-\widetilde{L}_1, \cdots, -\widetilde{L}_r\} .
\end{equation}
In this case 
it follows from Theorem \ref{cape} $(b)$ that $V_{1, m}$ is a finite union 
of torsion line bundles (because the corresponding subtori much have dimension zero).
Then the bundles $\widetilde{L}_1, \cdots, \widetilde{L}_r$ have finite order.  Therefore $L_1, \cdots, L_r$ have finite order in $\Pic^0 (Y')$

By \cite{BCHM}, we infer that 
$$\bigotimes^q F_m \rightarrow F_{qm} = p_\star (qm K_{X/Y} +qm \Delta)$$
is surjective for every $q\in\mathbb{N}$. Combining this with \eqref{directsumtrivial}, 
we obtain the surjective morphism
$$ \bigotimes^q (\bigoplus_{j=1}^r L_j) \rightarrow p_Y ^\star F_{qm}  \qquad\text{for every }q\in\mathbb{N} .$$
As $L_1, \cdots, L_r$ have finite order, by taking $q$ divisible enough we have 
$$h^0 (X', qm K_{X'}+ qm \Delta') \geq C_1 \rank F_{qm} \geq C_2 q^{\kappa (K_F +\Delta_F)}$$ 
for some $C_1, C_2 >0$ independent of $q$; we remark that $p_X : X'\rightarrow X$ depends on $m$ but is independent of $q$.
Then 
$$\kappa (X', K_{X'} + \Delta') \geq \kappa (K_F +\Delta_F) .$$
Since the Iitaka dimension is not affected by the finite \'etale cover, we have thus
$$\kappa (X, K_X + \Delta) \geq \kappa (K_F +\Delta_F) +\kappa (Y).$$

\medskip

\noindent $\bullet$ \emph{Case 3: The Kodaira dimension of $Y$ is greater than one, i.e. $\kappa(Y)\geq 1$.} In this case we use the Iitaka map
corresponding to the canonical bundle of $Y$. 

Let $m_0\gg 0$ be a large enough 
positive integer, such that the Iitaka map 
$$q_0: Y\dasharrow Z$$ of $K_Y$ is given by 
the linear system $\displaystyle |m_0K_Y|$. Here $Z$ is a smooth manifold, of dimension equal to $\kappa(Y)$. 

There exist a modification of $\pi_Y : Y' \rightarrow Y$ such that the induced 
map $q:= q_0\circ\pi_Y$ is an algebraic fiber space  
\begin{equation}\label{sg40}
q: Y'\rightarrow Z
\end{equation}
such that $\kappa (Y' _z)=0$, where $Y' _z:= q^{-1}(z)$ is a general fiber of $q$. 
The map $\pi_Y$ is such that the base points of the bundle $\displaystyle \pi_Y^\star(m_0K_Y)$
are divisorial, i.e. there exists an effective divisor $D$ on $Y^\prime$ such that 
\begin{equation}\label{sg41} 
\pi_Y ^\star(m_0 K_Y)= D+ L
\end{equation}
where $L$ is a line bundle generated by its global sections. Actually, the map 
$q$ in \eqref{sg40} is associated to the base-point free linear system $|L|$.
Therefore, for any point $z\in Z$ which is not a critical value of $q$, the 
fiber $Y^\prime_z$ is given by the equations
\begin{equation}\label{sg42}
\sigma_1=0,\dots,\sigma_d= 0
\end{equation}
where $d= \kappa(Y)$ and the $\sigma_j$ above are sections of $L$.

\smallskip

Let $X'$ be a desingularization of the fibered product $X\times_Y Y'$. 
We consider the natural projection 
$p' : X' \rightarrow Y'$; we can assume that the exceptional divisor of the map
\begin{equation}\label{sg21}
X' \to X\times_Y Y' 
\end{equation}
is contained in the $p^\prime$-inverse image of $Y^\prime\setminus Y_0$, where by definition $Y_0$ is
the Zariski open subset of $Y^\prime$ such that $\displaystyle \pi_Y|_{Y_0}$ is biholomorphic.
Let  
\begin{equation}\label{sg22}
X' _z :=(p^\prime)^{-1} (Y' _z)
\end{equation}
be the inverse image of $Y^\prime_z$; 
it is a non-singular manifold, provided that $z$ is general enough.
We have 
\begin{equation}\label{sg23}
K_{X'/Y'} + \Delta_{X'}+ E_Y= \pi_X^\star\big(K_{X/Y} + \Delta_{X}\big)+ E_X
\end{equation}
where $E_X$ is $\pi_X$-exceptional, $p'(E_Y)$ is $\pi_Y$-exceptional and $\displaystyle (X', \Delta_{X'})$ is klt.
We will denote by $\displaystyle F_m^\prime :=p' _\star (m K_{X'/Y'} +m\Delta_{X'})$ the direct image corresponding to 
 $p'$ and $\widetilde{g}_{X' /Y'}$ the corresponding metric on $F_m ^\prime$.
\begin{equation*} 
\begin{CD}
    X^\prime @>{\pi_X}>> X \\
    @Vp^\prime VV      @VVpV  \\ 
    Y^\prime @>{\pi_Y}>>   Y \\
     @VqVV\\
    Z \\
\end{CD}
\end{equation*} 
\smallskip

The arguments we will use for the rest of our proof are a bit more complicated than what they should be, because it is not clear that the bundle $\displaystyle F_m^\prime $ is flat.

\noindent In any case, we have the natural isometry
\begin{equation}\label{pullbackmor}
 i: \pi_Y ^\star F_m \rightarrow F_m^\prime   \qquad\text{on }Y_0 .
\end{equation}
The inverse image bundle $\pi_Y ^\star F_m$ is Hermitian flat on $Y'$. As $i$ is an isometry, by the 
argument at the end of the proof of Theorem \ref{flat}, given any flat local section $u$ of $\pi_Y ^\star F_m$
the image 
\begin{equation}\label{sg30}
i(u|_{Y_0})
\end{equation}
extends across $Y'\setminus Y_0$ as a section of $F_m^\prime$. In this way we obtain a morphism
\begin{equation}\label{sg31}
 i: \pi_Y ^\star F_m \rightarrow F_m^\prime 
\end{equation} 
defined globally on $Y'$.

\noindent Theorem \ref{conseqflat} will be established by using classical $L^2$-extension theorem and the following three lemmas; our goal is to show that we have 
\begin{equation}\label{yesman3}
\kappa (K_X +\Delta) \geq \kappa (K_F +\Delta_F) + \kappa(Y).
\end{equation}
\smallskip

\noindent The first step of the remaining part of our proof is as follows.

\begin{lemma}\label{kodaz}
Let $z\in Z$ be a general point. Then we have
\begin{equation}\label{sg90}
\kappa (Y'_z, K_{Y'_z}+ S|_{Y'_z} )= 0
\end{equation}
for any effective divisor $S$ on $Y'$ whose support is contained in the 
exceptional loci of $\pi_Y$.
\end{lemma}

\begin{proof}
Let
\begin{equation}\label{triv1}
K_{Y'}= \pi_Y^\star K_Y+ \Lambda
\end{equation}  
where $\Lambda$ is the exceptional divisor of $\pi_Y$. The properties of the Iitaka map show 
in particular that we have 
\begin{equation}\label{triv2}
h^0(Y'_z, l K_{Y^\prime _z})= h^0\big(Y'_z,l\big(\pi_Y^\star K_Y+ \Lambda)|_{Y^\prime _z} \big)= 1
\end{equation} 
for every $l$ sufficiently divisible.

Fix $u_0$ a non-zero section of some multiple of $K_Y$. The support of $S$ is contained in $\Supp(\Lambda)$,
and therefore it would be enough to show that we have
\begin{equation}\label{triv3}
h^0(Y'_z, r_1 K_{Y'_z}+ r_2 \Lambda|_{Y'_z})= 1
\end{equation}
for any $r_1 , r_2 \in\mathbb{N}$ sufficiently divisible. 
This is however immediate, because the sections in \eqref{triv3} can be injected 
in the space $H^0 (Y' _z, (r_1+r_2)K_{Y'_z})$ by multiplication with an appropriate power of $u_0$. Thanks to \eqref{triv2}, 
$h^0 (Y' _z, (r_1+r_2)K_{Y'_z}) =1$ and so the proof is finished.
\end{proof}
\medskip

\noindent The next result establish a very useful map between the sections of the bundles
$\displaystyle m\big(K_{X'/ Y'}+ \Delta_{X^\prime}\big)|_{X'_z}$ and 
$\displaystyle m\pi_X^\star(K_{X/Y}+ \Delta_{X})|_{X'_z}$, respectively.

\begin{lemma}\label{ext_gen} 
For every $L\in \Pic^0 (Y' _z)$, the morphism \eqref{sg31} induces an isomorphism
\begin{equation}\label{isomorge}
H^0 (Y' _z , \pi_Y ^\star F_m \otimes L) \rightarrow H^0 (Y' _z, F_m^\prime \otimes L) . 
\end{equation}
In particular, we have the natural isomorphism
\begin{equation}\label{isomor}
H^0 (Y' _z , \pi_Y ^\star F_m ) \rightarrow H^0 (Y' _z, F_m^\prime) 
\end{equation}
and
\begin{equation}\label{fibrq}
\kappa (K_{X'_z} +\Delta_{X^\prime}|_{X' _z}) \geq  \kappa (K_F +\Delta_F) .
\end{equation}
\end{lemma}

\begin{proof}
As $F_m$ is hermitian flat and $\kappa (Y'_z)=0$, by Theorem \ref{mathkang},
we have
\begin{equation}\label{sg74}
\pi_X^\star F_m  |_{Y^\prime _z} = \bigoplus_{i=1}^r L_i ,
\end{equation}
 modulo an \'etale cover of $Y'_z$, where $L_i \in \Pic^0 (Y^\prime _z)$ and $r$ is the rank of $F_m$.
In what follows, we still denote by $Y'_z$ the resulting manifold, in order to keep the notation as clean as possible --we observe that  
it is enough to establish the lemma for the pull-back of $u$. 

\smallskip

Let $u\in H^0 (Y' _z, F_m^\prime \otimes L)$. Thanks to \eqref{pullbackmor}, $u$ induces a section
$$\wh u :=i^{-1} (u)\otimes (\sigma_T |_{Y' _z})\in H^0 (Y' _z ,  \pi_X^\star F_m \otimes T \otimes L) ,$$
for some effective divisor $T$ supported in the exceptional loci of $\pi_Y$, where $\sigma_T$ is the canonical section of $T$.
Let $s_i$ be the components of section $\wh u$ according to \eqref{sg74}, and let $v_0 \in H^0 (Y' _z,  m_1 K_{Y' _z})$ for some 
$m_1 \in\mathbb{N}$ large enough.
Then we have 
\begin{equation}\label{decompelement}
v_0 \otimes s_i\in H^0\big(Y'_z,  m_1 K_{Y' _z} + T+ L_i +L \big). 
\end{equation}
By combining Lemma \ref{kodaz} with Theorem \ref{cape}, we infer that $L_i +L$ 
is a torsion point in $\Pic^0 (Y'_z )$ as soon as the corresponding section $s_i$ is non-zero.
Moreover, we have already a section 
$v_0 \otimes (\sigma_T |_{Y' _z} )\in H^0 (Y' _z, m_1 K_{Y' _z}+T)$.
Using again Lemma \ref{kodaz}, we infer that 
$$\Div (s_i) = [\sigma_T |_{Y' _z}] .$$ 
Therefore 
$$\Div (\wh u) =\Div (i^{-1} (u)\otimes (\sigma_T |_{Y' _z})) = [\sigma_T |_{Y' _z}] .$$
As a consequence, $\frac{\wh u}{\sigma_T} \in H^0 (Y' _z ,  \pi_X^\star F_m \otimes L)$
and \eqref{isomorge} is proved.

\smallskip

\noindent It remains to prove \eqref{fibrq}. 
Set 
$$V_{1, m}:=\{L\in \Pic^0 (Y' _z) | h^0 (X'_z , m K_{X'_z} + m\Delta_{X'} + q^\star L) \geq 1 \}.$$
Then \eqref{isomorge} and \eqref{sg74} imply that 
\begin{equation}\label{obs}
 V_{1, m} = \{- L_1, \dots, - L_r\}.
\end{equation}
Indeed, let $\displaystyle L\in \Pic^0 (Y'_z)$ be a line bundle 
in the set $V_{1, m}$. The equality \eqref{isomorge} shows that some of 
the bundles $L-L_j$ is effective; this establishes \eqref{obs}. 
 
By using the same argument as in the case 2 of the proof of 
Theorem \ref{conseqflat}, we deduce that $L_1 , \dots, L_r$ have finite order. 
The inequality \eqref{fibrq} is therefore proved.

\end{proof}
\medskip

\noindent The last technical statement which we need is the following.

\begin{lemma}\label{int}
Let $V$ be a section of the 
bundle $\displaystyle m\pi_X^\star(K_{X/Y}+ \Delta_{X})|_{X'_z}$. Then we have
\begin{equation}\label{sg77}
\sup_{X'_z}|V|^2e^{-\varphi_{X/Y}}< C< + \infty
\end{equation}
for some constant $C> 0$. Here $e^{-\varphi_{X/Y}}$ is the pull-back of the corresponding relative Bergman metric on $X$.
\end{lemma}

\begin{proof}
We consider a flat local frame for  $\pi^\star F_m$; by this we mean that we have local sections
$\sigma_1,\dots, \sigma_r$ of $\pi^\star F_m$ at a point $y_0\in Y_z'$ such that the quantity
\begin{equation}\label{sg78}
\int_{X_y}|\sigma_j^0|^2e^{-\varphi_{\Delta_X}- (m-1)\varphi^{(m)}_{X/Y}}
\end{equation}
is a (non-identically zero) constant with respect to $y$ in the complement of an analytic subset. In the notation above,
$\sigma^0_j$ is such that $\sigma^0_j\wedge d p= \sigma_j$.

\noindent By definition, the weight of the Bergman metric satisfies the pointwise estimate
\begin{equation}\label{sg79}
\varphi_{X/Y}(x)\geq 
\log\frac{|\sigma_j|^2}{\int_{X_y}|\sigma_j^0|^2e^{-\varphi_{\Delta_X}- (m-1)\varphi^{(m)}_{X/Y}}}
\end{equation}
where $y=  p(x)$ and combined with \eqref{sg78}, we obtain
\begin{equation}\label{sg80}
|\sigma_j|^2e^{-\varphi_{X/Y}}\leq \int_{X_y}|\sigma_j^0|^2e^{-\varphi_{\Delta_X}- (m-1)\varphi^{(m)}_{X/Y}}\leq C.
\end{equation}
The result follows, since we can express the section $V$ in terms of the flat generators $\sigma_j$.
\end{proof}
\medskip

\noindent We have now at our disposal all the ingredients needed to finish the proof of Theorem \ref{conseqflat}. Let 
$u$ be a section of the bundle $m\pi_X^\star(K_{X/Y}+ \Delta_{X})|_{X'_z}$, and let $\tau_0$ be a section of 
the bundle $m_0 K_Y$, where we assume that $m$ and $m_0$ are greater than 2. 
Thanks to \eqref{sg31}, the tensor product 
\begin{equation}\label{ext4}
\wh u:= u\otimes (p\circ \pi_X)^\star\tau_0 
\end{equation} induces a section of the bundle 
\begin{equation}\label{ext2}
K_{X'}+ \Delta_{X'}+ (m- 1)\pi_X^\star(K_{X/Y}+ \Delta_{X})+ (m_0-1)(\pi_Y\circ p')^\star K_Y|_{X'_z} .
\end{equation}
We endow the bundle $\displaystyle \pi_X^\star(K_{X/Y}+ \Delta_{X})$ with the $m^{\rm th}$ root of the metric
used in Lemma \ref{int}, and the bundle $\pi_Y^\star K_Y$ with the metric given by the decomposition
\eqref{sg41}. By Ohsawa-Takegoshi extension theorem, there exists a section $U$ of the bundle 
\begin{equation}\label{ext3}
K_{X'}+ \Delta_{X'}+ (m- 1)\pi_X^\star(K_{X/Y}+ \Delta_{X})+ (m_0 -1)(\pi_Y\circ p')^\star K_Y
\end{equation}
whose restriction to the fiber ${X'_z}$ is precisely $\wh u$ in \eqref{ext4}. Indeed, we are using 
here the fact that $X'_z$ is given by the inverse image of the equations \eqref{sg42}, together with Lemma \ref{int}.

Let $(u_j)$ be a basis of the space of global sections of the bundle 
$m\pi_X^\star(K_{X/Y}+ \Delta_{X})|_{X'_z}$ and let $(U_j)$ be the sections
of \eqref{ext3} obtained by the procedure explained above. If we denote by $(\tau_k)$  
a basis of the bundle $(m-m_0)K_Y$, then the family of sections 
$$U_j\otimes (\pi_Y\circ p')^\star \tau_k$$
are linearly independent. Moreover, they induce holomorphic sections of
the bundle $K_{X'} + \Delta_{X'} + (m-1) p_X ^\star (K_X+ \Delta_X)$. Combining this with \eqref{isomor} and \eqref{fibrq}, we get 
$$\kappa (K_{X'} +\Delta_{X'}  + E_Y) \geq \kappa (K_F +\Delta_F) + \kappa (Y) .$$
Notice that $E_Y$ is the exceptional loci of $\pi_X$, the proof of Theorem \ref{conseqflat} is thus finished.

\end{proof}
\medskip

\begin{remark} An interesting question related to the proof just finished is the following. We assume that the 
bundle $\det F_m$ is topologically trivial; it follows that $F_m$ itself is flat. If we consider 
a birational map $\pi_Y: Y^\prime\to Y$ and $p^\prime: X'\to Y^\prime$ as in \eqref{sg21}, does it follows that
$F^\prime_m$ is flat as well? This is clearly the case if $\pi_Y$ is finite instead of birational. 
\end{remark}
\medskip

\noindent Recently, the first named author of the present article has obtained the following result.

\begin{theorem}\label{jc_cm2}\cite{jc_cm2}
Let $p: X\rightarrow Y$ be an algebraic fiber space, where $Y$ is a projective variety of dimension at most two.
Let $\Delta$ be an effective $\bQ$-divisor such that the pair $(X, \Delta)$ is klt, and let $F$ be a generic fiber of $p$.
Then
\begin{equation}\label{}
\kappa (K_X +\Delta) \geq \kappa (K_F +\Delta_F)+ \kappa(Y), 
\end{equation}
where $\Delta_F = \Delta|_F$.
\end{theorem}
This settles the Iitaka conjecture in the case of a two-dimensional base.
The proof in \cite{jc_cm2} uses the techniques developed in the present article, together with some results in the theory of orbifolds of Calabi-Yau type.

\end{document}